\documentclass[11pt,final]{siamart251216}

\usepackage{amsfonts}
\usepackage{graphicx}
\usepackage[dvipsnames]{xcolor}
\usepackage{bm}
\usepackage{epstopdf}
\usepackage{enumitem}
\usepackage{soul} 
\usepackage{csquotes} 
\usepackage{bbm}
\usepackage{makecell}
\usepackage{cite}

\newsiamthm{assumption}{Assumption}
\newsiamthm{remark}{Remark}

\usepackage{tikz}
\usetikzlibrary{patterns,positioning}
\usepackage{pgfplots}
\pgfplotsset{compat=newest}
\usepgfplotslibrary{fillbetween,groupplots}
\usetikzlibrary{arrows,shapes,trees,calc}
\usetikzlibrary{decorations,decorations.markings}

\usepackage{algpseudocode}

\newfloat{algorithm}{h}{loa}
\floatname{algorithm}{Algorithm}

\definecolor{col1}{HTML}{BBBBBB}
\definecolor{col2}{HTML}{009988}
\definecolor{col3}{HTML}{CC3311}
\definecolor{col4}{HTML}{EE3377}
\definecolor{col5}{HTML}{33BBEE}
\definecolor{col6}{HTML}{0077BB}
\definecolor{col7}{HTML}{EE7733}

\title{Optimal experimental design for passive imaging source problems\thanks{Version from \today}\funding{All authors acknowledge support from the DFG through Grant 432680300 -- SFB 1456 (C04). C.~A.~and G.~S.~were partially supported by the Multidisciplinary University Research Initiatives (MURI) Program, Office of Naval Research (ONR) grant \#N00014-19-1-242, and by ONR \#N00014-26-1-2101. Moreover, C.~A.~acknowledges support from the Initiative d'Excellence d'Aix-Marseille Université -- A*MIDEX, AMX-21-RID-022 during the completion of the final manuscript.
}}

\author{Christian Aarset\thanks{Institute for Numerical and Applied Mathematics, Georg-August Universität Göttingen, Göttingen, Germany, \email{c.aarset@math.uni-goettingen.de}, \email{hohage@math.uni-goettingen.de}} \and Thorsten Hohage\footnotemark[2]
\and  Georg Stadler\thanks{Courant Institute School of Mathematics, Computing and Data Science, New York University, New York, USA, \email{stadler@cims.nyu.edu}}}

\date{\today}

\newcommand{\Epsilon}{{E}}

\newcommand{\C}{\mathbb{C}}
\newcommand{\E}{\mathbb{E}}
\newcommand{\N}{\mathbb{N}}
\renewcommand{\P}{\mathbb{P}}
\newcommand{\R}{\mathbb{R}}

\newcommand{\Ac}{\mathcal{A}}
\newcommand{\Bc}{\mathcal{B}}
\newcommand{\Cc}{\mathcal{C}}
\newcommand{\Fc}{\mathcal{F}}
\newcommand{\Jc}{\mathcal{J}}
\newcommand{\Kc}{\mathcal{K}}

\newcommand{\Nc}{\mathcal{N}}
\newcommand{\Qc}{\mathcal{Q}}

\newcommand{\eps}{{\bm{\epsilon}}}

\newcommand{\e}{{\bm{e}}}

\newcommand{\g}{{\bm{g}}}
\newcommand{\q}{{\bm{q}}}

\newcommand{\w}{{\bm{w}}}
\newcommand{\x}{{\bm{x}}}
\newcommand{\y}{{\bm{y}}}
\newcommand{\z}{{\bm{z}}}

\newcommand{\B}{{{B}}}
\newcommand{\M}{{{M}}}

\newcommand{\ellh}{{\widehat{\ell}}}

\newcommand{\Ch}{{\widehat{C}}}

\newcommand{\Fch}{{\widehat{\Fc}}}

\newcommand{\Lh}{{\widehat{L}}}
\newcommand{\Lt}{{\widetilde{L}}}

\newcommand{\Qh}{{\widehat{Q}}}
\newcommand{\Rh}{{\widehat{R}}}

\let\d\undefined
\let\vec\undefined
\DeclareMathOperator{\d}{d\!}
\DeclareMathOperator{\tr}{tr}

\DeclareMathOperator{\Cov}{Cov}
\DeclareMathOperator{\diag}{diag}
\DeclareMathOperator{\Diag}{Diag}
\DeclareMathOperator{\mat}{mat}
\DeclareMathOperator{\vec}{vec}

\DeclareMathOperator{\HS}{HS}

\DeclareMathOperator*{\argmin}{argmin}
\newcommand{\post}{{\mathop{\mathrm{post}}}}

\newcommand{\corr}{\otimes}
\newcommand{\kron}{\otimes}
\newcommand{\schur}{\odot}
\newcommand{\khra}{\ast}

\newcommand{\dua}[3][X]{\left({#2},{#3}\right)_{#1}}

\newcommand{\inner}[3][X]{\left\langle {#2},{#3}\right\rangle_{#1}}

\newcommand{\Gmn}{\Gamma}
\newcommand{\Gmnwi}{\Gamma^{\dagger}_{\!\w}}

\def\measurehat#1{%
   \setbox0=\vbox{$\widehat{#1}\hfil\break$\null\par
      \setbox0=\lastbox\unskip\unpenalty\global\setbox1=\lastbox}%
   \setbox0=\hbox{\unhbox1 \unskip\unpenalty\unskip \global\setbox2=\lastbox}%
   \setbox0=\vbox{\unvbox2 \setbox0=\lastbox}%
}
\def\doublehat#1{%
   \measurehat{#1}\dimen0=\wd0 \measurehat{\kern0pt#1}%
   \raise.35ex\rlap{\kern\dimexpr\dimen0-\wd0$\widehat{\phantom{#1}}$}{\widehat#1}%
}

\newcommand{\Mw}[1][\w]{\Diag(#1)}

\newcommand{\wpcontnom}{\overline{\w}}

\newcommand{\wgreedynom}{{\wpcontnom}_\mathrm{greedy}}

\usepackage{datatool,fp}

\newcommand{\DTLfetchsave}[5]{%
  \edtlgetrowforvalue{#2}{\dtlcolumnindex{#2}{#3}}{#4}%
  \dtlgetentryfromcurrentrow{\dtlcurrentvalue}{\dtlcolumnindex{#2}{#5}}%
  \let#1\dtlcurrentvalue
}

\begin{document}

\DTLloaddb
  {posteriors}
  {CSV/posteriors.csv}
    
\DTLfetchsave\posterioronemax{posteriors}{design}{ones}{posteriormax}
\DTLfetchsave\posteriorzeromax{posteriors}{design}{zero}{posteriormax}

\DTLdeletedb{posteriors}

\maketitle

\begin{abstract}
This work focuses on optimal experimental design (OED) methods for passive imaging. We adopt a Bayesian inverse problem framework for passive imaging source problems, primarily focusing on spatially uncorrelated sources and systems governed by the Helmholtz equation. A major challenge in passive imaging is that the use of correlation data causes the observation dimension to grow quadratically with the number of sensor locations, compounding the computational difficulty of finding optimal designs. To overcome the computational bottleneck of repeated PDE solves in optimal design algorithms, we develop a two-level, low-rank approximation of the A-optimal design objective. This effectively decouples the problem into an offline and an online phase, enabling efficient evaluation of the design objective and its gradient without additional PDE solves. Our numerical results demonstrate that the proposed algorithm efficiently scales to large problems and that the resulting optimal designs significantly outperform random sensor placements in minimizing posterior uncertainty.

\end{abstract}



\section{Introduction}
Passive imaging infers parameters from the covariance (or correlation) of measurements made at multiple sensors. Because the data are a covariance matrix, the effective data dimension grows quadratically with the number of sensors, in contrast to the linear growth of conventional (i.e., ``active'') inverse problems. This dimension squaring is a central computational challenge in passive imaging, and it makes the choice of sensor locations, the \emph{optimal experimental design (OED) problem}, particularly important. 

This work presents an OED approach to passive imaging source identification problems. In analogy with earlier studies on active imaging, we formulate the design problem within a Bayesian framework. Our main contributions are the definition of an A-optimal utility tailored to passive imaging, and a two-level low-rank decomposition of this utility that separates design-independent computations (an offline phase requiring PDE solves) from design-dependent computations (an online phase whose cost does not scale with the governing PDE). The resulting algorithm can efficiently determine A-optimal sensor layouts for two-dimensional Helmholtz problems with hundreds of possible sensor locations and yields sensor designs that significantly outperform random placements.

Unlike the more common setting of active imaging and inverse problems, passive imaging focuses on identifying properties of an unknown ambient noise source or random medium using a large number of observations generated with different, unknown sources. Passive imaging is a vibrant research area \cite{GarPap16}, with many applications in the natural sciences including geophysics, helioseismology, and aeroacoustics. In aeroacoustics, a standard problem \cite{Ste22} is to reconstruct a noise source in an aircraft or part thereof from pressure measurements collected by nearby sensors. The acoustic source is represented as a random variable with zero mean, and thus, the resulting linear acoustic wave also has zero mean fluctuations. Consequently, inferring the expected source is not informative. Instead, its (co)variance, i.e.,~the source strength, provides useful information. The linear relationship between the source and the measured data makes this a covariance inference problem, in which one estimates the source covariance from the empirical covariance of the observations. This leads to the dimension squaring characteristic of correlation-based data, rendering passive imaging a challenging practical problem.

\subsection{Related literature}\label{ssec:related_literature}

Random inverse source problems with correlation data, as studied in this 
paper, have applications in noise localization in aeroacoustics; see, e.g.,  
\cite{devaney:79,HRS:20,li2023stability}. Such problems also appear 
in imaging in random media where random fluctuations of the medium 
can be modeled as virtual sources \cite{GarPap16}. 
Often, not only the source strength is unknown but also other parameters, 
such as coefficients of the differential operator, giving rise to 
nonlinear passive imaging problems. This is the case, 
e.g., in local helioseismology, \cite{duvall_etal:93, Aga20,MulHohFouGiz23}, 
where turbulent convection is the dominant source
of acoustic waves, but also in earth seismology \cite{schuster2009seismic} and in underwater acoustics using ambient noise as a source \cite{BSS:11}. 

Optimal experimental design for traditional (i.e., active) imaging and related inverse problems has attracted significant interest in recent years \cite{Ale21, Huan_Jagalur_Marzouk_2024}.
Yet, to the best of the authors’ knowledge, there have not been any systematic efforts to investigate OED for passive imaging in a comparably general setting. This may have to do with the substantial computational challenges arising in passive imaging and potentially with the difficulty of analyzing stochastic aspects of the correlation data. In fact, even under otherwise Gaussian assumptions, correlation data follows a conditional Wishart distribution \cite{Wis1928,Mui82} rather than a Gaussian distribution.

Optimal design methods for high- or infinite-dimensional active inverse problems have largely depended on a combination of linearization or related approximations applied to both the inverse problem and the utility objective, together with carefully constructed low-rank matrix and operator approximations and machine-learned surrogate models, in order to keep the computational cost tractable
\cite{AlePetStaGha14, AreCheDegVer24, WuKeyOleThoCheGha23, Aarset2024,AttCon2022}. For passive imaging, these challenges are even more severe due to the high data dimension. We therefore focus on aggressively reducing model and data dimensions using matrix analysis, low-rank approximations, and by restricting attention to problems with spatially uncorrelated random sources.

\subsection{Contributions and limitations}
Here, we make the following main contributions.

(1) We formulate the optimal design problem for a Bayesian passive imaging source inversion. Formulations for both correlated and uncorrelated noise are presented, and we develop an optimal design algorithm for the case of uncorrelated noise.
(2) We develop a two-level low-rank decomposition of the A-optimal objective that confines all PDE solves to a one-time offline phase; the design optimization loop runs entirely on precomputed factors, independent of the PDE governing the parameter-to-observation map.
(3) We present a comprehensive numerical study for passive imaging problems governed by the two-dimensional Helmholtz equation, mimicking typical settings in aeroacoustics.

The present approach has the following limitations.

(1) We assume an infinite number of uncorrelated noise realizations by requiring that the correlation data noise is Gaussian.
(2) The passive imaging source problem we study is linear. Nonetheless, the quadratic growth of the data dimension compared with direct/active inverse problems poses a significant challenge, even when one only aims to solve passive imaging problems with \emph{fixed} sensors.
(3) Even with the two-level compression applied in the offline phase, the problem is still high-dimensional, and extending the method to compute optimal designs for three-dimensional passive imaging would demand substantial computational resources.

\subsection{Notation and assumptions}\label{ssec:notation}
Throughout, $X$ denotes a Hilbert space satisfying the Gelfand triple relation
$
    X^*\subset L^2(\Omega)\subset X
$
for an open set $\Omega\subset\R^d$, $d\in\N$, where the star for spaces denotes the dual space, while for operators, the star denotes the Banach adjoint. We write $L(X^*,X)$ for the space of bounded linear operators from $X^*$ to $X$. Within this space, we define $\HS(X^*,X)\subset L(X^*,X)$ (and similarly its dual, $\HS(X,X^*)$) to be the subspace consisting of Hilbert–Schmidt operators, equipped with the duality pairing
\begin{equation}\label{eq:HS_pairing}
    \dua[\HS]{\Ac}{\Bc} := \sum_{i=1}^\infty
    \left(
        \Ac e_i,\Bc e_i
    \right) = \tr(\Bc^*\Ac) = \tr(\Ac^*\Bc),
\end{equation}
for $\Ac\in\HS(X^*,X)$, $\Bc\in\HS(X,X^*)$, where $\{e_i\}_{i=1}^\infty\subset X^*$ is any $L^2(\Omega)$-orthonormal basis. Similarly, the Hilbert space $\R^{m\times m}$, $m\in\N$ is equipped with the standard trace inner product. Finite-dimensional vectors are denoted by lowercase bold symbols (e.g.,~$\w$, $\g$), function-space variables and scalar quantities by lowercase letters (e.g.,~$f$, $q$), real matrices by uppercase letters (e.g.,~$G$, $Q$, $M$), and operators acting on function spaces and/or mapping into them are written in calligraphic font (e.g.,~$\Fc$, $\Qc$). The notation $\widehat{\cdot}$ denotes operators or matrices associated with correlation data (e.g.~$\widehat{\Fc}$, $\Qh$).
Integers $m,n$ denote the dimensions of the observed data and the discretization, respectively, and $N$ is the number of samples.

\paragraph{Matrix operators} 
We will repeatedly use the following matrix operations:
\begin{itemize}[leftmargin=3ex]
    \item ``$\vec$'', which maps a matrix to the vector containing its entries, ordered column-wise starting in the top left, and ``$\mat$'', the inverse operation, which maps a vector to the matrix containing all its entries.
    \item The operator ``$\Diag$'' takes a vector and returns a diagonal matrix whose diagonal entries are given by that vector. Its corresponding left inverse, denoted by ``$\diag$'', takes a matrix and returns the vector formed by its diagonal elements. We will also employ these operators in a function space setting, where they are defined as natural extensions of their finite-dimensional counterparts.
    \item The Kronecker product $\kron$ is interchangeably used in the following ways:
    \begin{itemize}[leftmargin=4ex]
        \item[$\diamond$] As the classical Kronecker product operation between two matrices
            \begin{align*}
                A\kron B & = 
                \begin{bmatrix}
                    AB_{11} & AB_{12} & \ldots \\
                    AB_{21} & AB_{22} & \ldots \\
                    \vdots & \vdots & \vdots
                \end{bmatrix},
            \end{align*}
        where we used matrix block notation.
        \item[$\diamond$] If $A:X_A\to Y_A$ and $B:X_B\to Y_B$ are linear mappings, $A\kron B: L(X_A,X_B^*)\to L(Y_A,Y_B^*)$ is defined by  
        $C\mapsto ACB^*$. If $A$ and $B$ are matrices, this is compatible 
        with the Kronecker product notation since 
        $(A\kron B)\x := \vec(A\mat(\x)B^T)$.
        \item[$\diamond$] $f\kron f' := \inner{\cdot}{f'}f \colon X^*\to X$ for general $f$, $f'\in X$, which for finite-dimensional $X$ reduces to the outer product. Note that for $A,A'\in L(X,Y)$ we have 
        $(A\kron A')(f\kron f')=Af \kron A'f'$.
    \end{itemize}
    \item The Schur product ``$\schur$'' is defined via $(A\schur B)_{ij}=A_{ij}B_{ji}$ for compatible matrices $A$, $B$, and all indices $i$, $j$, noting the interchanged indices on $B$ to match the Kronecker product definition. 
    \item The (column-wise) Khatri-Rao product ``$\khra$'' is such that its $i$-th column is the Kronecker product of the $i$-th column of the two input matrices, that is,
    \[
    	\left(A \khra B\right)_{i+(k-1)n,j} = A_{kj}B_{ji}
    \]
    for indices $i$, $j$, $k$ and compatible matrices $A$, $B$ with $n$ rows.     
\end{itemize}
Finally, $A^{1/2}$ denotes the matrix half-power for various diagonal positive semi-definite matrices $A$. 

\section{Passive imaging}\label{sec:passive_imaging}

Our starting point is the standard inverse problem relation%
\begin{equation}\label{eq:IP_uncorrelated}
    \g = \Fc f + \eps, \quad \eps \sim\Nc(0,\Gmn),
\end{equation}
where $f\in X$ is an unknown parameter, $\Fc \colon X\to\R^m$ the forward operator, and $\g\in \R^m$ is the measured data. What distinguishes passive imaging is that the parameter $f\in X$ is itself a random variable rather than a fixed unknown. We assume $f\sim\Nc(0,\Qc)$ with $\Qc\in \HS(X^*,X)$. Since the mean of $f$ is known (and zero), our object of interest is instead the covariance $\Qc$, which fully characterizes the distribution of $f$.
Since the mean of $\g$ is also zero, no information about $\Qc$ is contained in the empirical mean of $\g$, and the informative summary of $\g$ is the covariance $\Cov[\g] = \Fc\Cov[f]\Fc^* +\Gmn= \Fc \Qc\Fc^*+\Gmn$, where $\Fc^*$ is the adjoint operator. This converts an inverse problem with vector data into one with matrix data. In particular,
given an estimator $G\in \mathbb R^{m\times m}$ of $\Fc \Qc \Fc^*$, we define the \emph{passive imaging} problem: recover $\Qc \in \HS(X^*,X)$ from the matrix data $G\in\R^{m\times m}$ using the linear operator relation%
\begin{equation}\label{eq:IP}
    G = \Fc \Qc \Fc^* + E,
\end{equation}
where $E\in \R^{m\times m} $ is a symmetric, centered matrix-valued error with covariance matrix $\Gmn\kron\Gmn$. The forward map $\Qc \mapsto \Fc\Qc\Fc^*$ is linear; however, in contrast to \eqref{eq:IP_uncorrelated}, it involves $\Fc$ on both sides and the data now live in dimension $m^2$ instead of $m$. These distinctions underlie both the analytical considerations and the computational difficulties addressed in the rest of this paper.

In practice, $G$ is computed as a sample covariance matrix from $N$ independent realizations $\g^{(i)}$ of $\g$, i.e., $G := \frac{1}{N}\sum_{i=1}^N \g^{(i)} \kron \g^{(i)} - \Gmn$. The unscaled sample covariance $\sum_{i=1}^N \g^{(i)}\kron \g^{(i)}$ follows a Wishart distribution with $N$ degrees of freedom and scale matrix $\Fc \Qc\Fc^* +\Gmn$. The $G$ above is a centered and $1/N$-scaled version (see \cref{app:Wishart}).
To avoid working directly with this Wishart distribution, we adopt the Gaussian approximation  
$G|\Qc\sim\Nc(\Fc \Qc\Fc^*,\Gmn\kron\Gmn)$. The Gaussian shape is justified by the central limit theorem for sample covariances (see \cref{corr:wishart_inverse}). The specific choice of $\Gmn\kron\Gmn$ for the covariance is a simplification that preserves analytic tractability; the true asymptotic covariance depends on $\Qc$ and is discussed in \cref{app:Wishart}.

In many settings, the forward operator $\Fc$ produces complex-valued outputs, such as in aeroacoustics discussed in \cref{sec:Helmholtz}. For convenience, we consistently transform our data using a complex-to-real mapping
\begin{equation}\label{eq:complex_stack}
    \widehat{\g}\in\C^m\mapsto \g:=[\Re\widehat{\g},\Im\widehat{\g}]\in\R^{2m}.
\end{equation}
For general zero-mean, complex-valued Gaussian vectors $\widehat{\g}$, the real covariance matrix $\Cov[\g]$ is a complete descriptor of the distribution of $\widehat{\g}$. In contrast, the complex covariance matrix $\Cov[\widehat{\g}]:=\E[\widehat{g}\otimes\bar{\widehat{g}}]$ only fully describes circular complex Gaussian vectors. For non-circular Gaussian vectors the relation matrix $\mathrm{Rel}[\widehat{g}]=\E[\widehat{g}\otimes\widehat{g}]$ does
not vanish, and it is straightforward to verify that the correspondence between $(\Cov[\widehat{\g}],\mathrm{Rel}[\widehat{\g}])$ and $\Cov[\g]$ is an isomorphism, see \cite{picinbono:96}. This justifies our restriction to real-valued data.

\subsection{Bayesian formulation of passive imaging}
To formulate a corresponding Bayesian inverse problem for the operator $\Qc$, assume a prior distribution $\Qc\sim\Nc(\Qc_0,\overline{\Cc}_0)$ with $\Qc_0\in\HS(X^*,X)$ as the mean of this distribution, and appropriate covariance operator $\overline{\Cc}_0\in L(\HS(X,X^*),\HS(X^*,X))$. We treat $\Fc\kron\Fc\colon \HS(X^*,X)\to\R^{m\times m}$ as the forward map, whose Banach adjoint is $\Fc^*\kron \Fc^* \colon \R^{m\times m}\to \HS(X,X^*)$, as we will show later in \cref{prop:adjoint}. Further, we equate $\R^{m\times m}\simeq \R^{m^2}$. Then the Bayesian linear inverse problem corresponding to \eqref{eq:IP} allows an explicit posterior distribution $\Qc\mid G\sim\Nc(\Qc_\post,\overline{\Cc}_\post)$, where
\begin{equation}\label{eq:posterior}
\begin{split}
    \Qc_\post & := \Cc_\post \left(\Fc^*\Gmn^{-1}G + \overline{\Cc}_0^{-1}\Qc_0
    \right), \\
    \overline{\Cc}_\post & := \left(
        \right[\Fc^*\Gmn^{-1}\Fc\left] 
            \kron
        \right[\Fc^*\Gmn^{-1}\Fc\left] +
        \overline{\Cc}_0^{-1}
    \right)^{-1},
\end{split}
\end{equation}
see \cite{Stu10, DaPrato06}. Here, $\overline{\Cc}_0$ and $\overline{\Cc}_\post$ are covariances between spaces of operators. The approximation of and performing practical computations with \eqref{eq:posterior} is challenging due to the high dimension of the data and the uncertain parameter. This is especially true when the unknown $f$ is in an infinite-dimensional Banach space $X$, so that the parameter is $\Qc\in \HS(X^*,X)$. Moreover, if the data $\g$ are in $\R^m$, then $G\in \R^{m\times m}\simeq\R^{m^2}$. To address this, we henceforth focus on the case where the source covariance $\Qc$ is parametrized by a much lower-dimensional function space variable $q$.

\subsection{Spatially uncorrelated sources}\label{subsec:uncorrelated}
In the context of aeroacoustics, one simplification that avoids the squaring of the parameter dimension compared to direct observations is to assume that the random source $f$ is \emph{spatially uncorrelated}. Uncorrelated sources can be seen as the infinite-dimensional analogue of $f$ having a diagonal covariance matrix. In \cite{MulHohFouGiz23}, the authors demonstrate that for $X:=H^{-s}(\Omega)$, $\Omega\subset\R^d$, $s>d/2$, satisfying $X^*=H^s_0(\Omega)$, there exists a non-trivial class of random sources $f$ on $H^{-s}(\Omega)$ whose covariance operator $\Qc$ can be expressed as the multiplication operator $\Diag(q)\in \HS(H^{s}_0(\Omega),H^{-s}(\Omega))$, $q\in L^\infty(\Omega)$, $q\geq 0$ a.e.~in $\Omega$, that is, $[\Diag(q)\varphi](x):=q(x)\varphi(x)$ for a.e.~$x\in\Omega$ and all $\varphi\in H_0^s(\Omega)$. 

Under spatial uncorrelatedness, the operator $\Diag: L^\infty(\Omega)\to\HS(X^*,X)$ naturally enters the parameter-to-observable operator $\Fch\colon L^\infty(\Omega)\to\R^{m\times m}$ of the passive imaging problem, defined by
\begin{equation}\label{prop:adjoint:eq:forward}
    \Fch q := \Fc\Diag(q)\Fc^* \in\R^{m\times m}.
\end{equation}
To study inverse problems with $\Fch$ as the forward map, we will require the Banach adjoint $\diag \colon   \HS(X,X^*)\to\left(L^\infty(\Omega)^*\right)$ of $\Diag$. As an infinite-dimensional analogue of the matrix-to-vector diagonal operator introduced in \cref{ssec:notation}, $\diag$ can be understood as returning the \enquote{diagonal} of the kernel of sufficiently smoothing integral operators; a full characterization is given in the below lemma.

\begin{lemma}\label{lemma:diag}

Given $\Kc\in\HS(X,X^*)$ of the form
\[
    \Kc \xi = \int_\Omega k(\cdot,\y)\xi(\y)\d \y \quad \text{for } \xi\in X,
\]
one has $\diag\Kc\in L^1(\Omega)\subset \left(L^\infty(\Omega)\right)^*$, and the following properties, which are a direct consequence of \cite[Prop.~3, Lem.~6]{MulHohFouGiz23} and \cite[{\S}29]{Mercer1909}, hold:

\begin{enumerate}
    \item $[\diag\Kc](\x) = [\Kc\delta_{\x}](\x)=k(\x,\x)$ for a.e.~$\x\in\Omega$.
    \item Assume $(\lambda_i,e_i)_{i=1}^\infty\subset \R\times X^*$ is an eigensystem of $\Kc$, in the sense that $\Kc e_i = \lambda_ie_i$ for all $i\in\N$ under the embedding $X^*\hookrightarrow X$. Then
    \[
        \diag\Kc = \sum_{i=1}^\infty \lambda_i e_i^2.
    \]
\end{enumerate}
In particular, $\tr\Kc := \sum_{i=1}^\infty\lambda_i = \int_\Omega [\diag\Kc](\x)\d\x$.
\end{lemma}
This characterization enables us to study the adjoint of $\Fch$.

\begin{lemma}\label{prop:adjoint}
The Banach adjoint $\Fch^* \colon \R^{m\times m}\to \left(L^\infty(\Omega)\right)^*$ is given as
\begin{equation}\label{prop:adjoint:eq:adjoint}
    \Fch^* G = \diag\left(
        \Fc^*G\Fc
    \right)\quad \text{for $G\in \R^{m\times m}$.} 
\end{equation}
Moreover, $\Fch^* G\in L^1(\Omega)\subset \left(L^\infty(\Omega)\right)^*$ for all $G\in\R^{m\times m}$, and for $\bm{g}\in\R^m$, one has
\begin{equation}\label{prop:adjoint:eq:L1diag}
    \left[\Fch^*[\bm{g}\kron\bm{g}]\right](x) = \left[\Fc^*\bm{g}\right]^2(x) \qquad \text{for a.e.~$x\in\Omega$.}
\end{equation}

\end{lemma}

\begin{proof}

Adjointness of the maps $G\in\R^{m\times m}\mapsto \Fc^*G\Fc\in\HS(X,X^*)$ and $\Qc\in\HS(X^*,X)\mapsto \Fc \Qc\Fc^*\in\R^{m\times m}$ follows from
\begin{align*}
    \left\langle
        G,\Fc \Qc\Fc^*
    \right\rangle_{\R^{m\times m}} & =
    \tr\left(
        \Fc \Qc^*\Fc^*G
    \right) = 
    \tr\left(
        \Qc^*\Fc^*G\Fc
    \right) \\
    & = 
    \dua[\HS]{\Qc}{\Fc^*G\Fc}
\end{align*}
via cyclicity of the trace and \eqref{eq:HS_pairing}. \cref{lemma:diag} now implies \eqref{prop:adjoint:eq:adjoint}. For \eqref{prop:adjoint:eq:L1diag}, note
\begin{align*}
    \left[\Fc^*\left[\g\kron \g\right]\Fc\right]\varphi & = 
    \left\langle \g,\Fc\varphi\right\rangle_{\R^m} \Fc^*\g
    = \left( \Fc^*\g,\varphi\right)_{X^*,X} \Fc^*\g \\
    & = \int_{\Omega}\left[\Fc^*\g\right](\cdot)\left[\Fc^*\g\right](\y)\varphi(\y)\d \y
\end{align*}
for $\varphi\in X$. The remaining claims are now consequences of \cref{lemma:diag}.
\end{proof}

Note that \eqref{prop:adjoint:eq:L1diag} allows us to compute the back-propagated correlation data $\Fch^*G$ in essentially the same way as in \cite{MulHohFouGiz23, Ngu2025} when $G$ is a sample covariance matrix, i.e., $G = \frac{1}{N}\sum_{i=1}^N\bm{g}^{(i)}\kron\bm{g}^{(i)}$ for random data samples $(\bm{g}^{(i)})_{i=1}^N$. This circumvents the need to explicitly form the covariance matrix $G$ and apply the two-sided PDE operator $\Fc^*\kron\Fc^*$. In fact, $G$ never has to be constructed: instead, one can back-propagate each individual random sample $\g^{(i)}$ through $\Fc^*$, incrementally accumulate the resulting contributions, and then discard $\g$, thereby lowering memory requirements.

We now slightly broaden the discussion on complex-valued data. Suppose we aim to perform backpropagation via \cref{prop:adjoint} using complex data samples $(\widehat{\g}^{(i)})_{i=1}^N$ and their complex-to-real converted counterparts $(\g^{(i)})_{i=1}^N := ((\Re\widehat{\g}^{(i)},\Im\widehat{\g}^{(i)}))_{i=1}^N$. In light of \eqref{prop:adjoint:eq:L1diag} and the arguments given around \eqref{eq:complex_stack}, it is in general neither adequate nor memory-efficient to explicitly form the complex sample covariance matrix $\tfrac{1}{N}\sum_{i=1}^N\widehat{\g}^{(i)}\kron\overline{\widehat{\g}^{(i)}}$. Although one can simply construct the real sample covariance matrix $G$ as above, or, equivalently, form both the complex sample covariance matrix and the complex sample relation matrix $\tfrac{1}{N}\sum_{i=1}^N\widehat{\g}^{(i)}\kron\widehat{\g}^{(i)}$, this approach is considerably less memory efficient than first backpropagating and then accumulating each individual complex or, respectively, real sample using \eqref{prop:adjoint:eq:L1diag}.
 
\subsection{Bayesian passive imaging with uncorrelated sources}
Next, we define the Bayesian inverse problem and derive the explicit form of the posterior. The relation between parameters $q\in L^\infty(\Omega)$ and data $G\in \mathbb R^{m\times m}$ is given by
\begin{equation}\label{eq:ip_diagonal}
  G = \Fc\Diag(q)\Fc^* + \Epsilon,
\end{equation}
where the error is $E\sim\mathcal N(\boldsymbol 0,\Gmn\kron\Gmn)$. It remains to specify a prior distribution for the covariance diagonal $q$. While $q$ should be at least non-negative as the diagonal of a covariance operator, this is not necessary for \eqref{eq:ip_diagonal} to be well-defined. Hence, we can assume a Gaussian prior $q\sim\Nc(q_0,\Cc_0)$, $q_0\in X$, $\Cc_0\in\HS(X^*,X)$, as long as realizations of this distribution are almost surely in $L^\infty(\Omega)$, which follows for typical priors from Fernique's theorem \cite{da2014stochastic}.  One observes the decrease in dimensionality of the prior -- and thus also the posterior -- compared to that required for \eqref{eq:posterior}.

\begin{proposition}\label{prop:posterior}
The posterior distribution of the Bayesian inverse problem detailed above is $q \mid G\sim\Nc(q_\post,\Cc_\post)$, where
\begin{equation}\label{prop:posterior:eq:posterior_diagonal}
\begin{split}
    q_\post & := q_0 + \Cc_\post \diag\left(
    \Fc^*\Gmn^{-1}\left(
        G - \Fc \Diag(q_0) \Fc^*
    \right)
    \Gmn^{-1}\Fc\right), \\
    \Cc_\post & := \left(
        \diag \circ \left[\left[\Fc^*\Gmn^{-1}\Fc\right] 
            \kron
        \left[\Fc^*\Gmn^{-1}\Fc\right]\right] \circ \Diag +
        \Cc_0^{-1}
    \right)^{-1}.
\end{split}
\end{equation}

\end{proposition}

\begin{proof}
Using $\Fch:=\left[\Fc\kron\Fc\right]\circ\Diag \colon L^\infty(\Omega)\to\R^{m\times m}$, we have
\[
    \Fch^* = \diag\circ\left[\Fc^*\kron\Fc^*\right] \colon \R^{m\times m}\to L^1(\Omega)\subset \left(L^\infty(\Omega)\right)^*
\]
by \cref{prop:adjoint}. As $\Cov[G\mid q]^{-1}=\left(\Gmn\kron\Gmn\right)^{-1} = \Gmn^{-1}\kron \Gmn^{-1}$ and
\[
    \left[\Fc^*\kron\Fc^*\right] \circ \left[\Gmn^{-1}\kron\Gmn^{-1}\right] \circ \left[\Fc\kron\Fc\right] 
    =
    \left[\Fc^*\Gmn^{-1}\Fc\right] 
            \kron
        \left[\Fc^*\Gmn^{-1}\Fc\right],
\]
the result is now immediate from \cite{Stu10}. 
\end{proof}

\section{Optimal designs for passive imaging with uncorrelated sources}\label{sec:OED}

We now turn our attention to the optimal experimental design problem using the A-optimality criterion as the utility function. In this section, we formulate the passive imaging A-optimal design objective (equation \eqref{eq:A_optimal}), study the discretization of the parameter space and objective (\cref{ssec:FEM,ssec:A-opt-discrete}), and develop low-rank approximations to make evaluation of the objective computationally efficient (\cref{ssec:low-rank}); the resulting computational complexities are studied in \cref{ssec:complexities}.

We assume that the noise covariance $\Gmn$ is diagonal, and model sensor placement by appending a binary mask $\w\in\R^m$ to the finite observation operator in \eqref{eq:IP_uncorrelated}. A vector entry $w_k=1$ represents placing a sensor to observe the $k$-th component of the full measurement $\g\in\R^m$, while $w_k=0$ corresponds to not doing so. This leads to the design-dependent formulation
\begin{equation}\label{eq:IP_uncorrelated_sensor}
    \g = \Fc_\w f + \eps,    
\end{equation}
where $\Fc_\w \colon X\to\R^m$ is defined by $\Fc_\w := \Diag(\w)\Fc$. This framework directly generalizes to the passive imaging setting, in which one seeks to solve the design-dependent passive imaging problem
\begin{equation}\label{eq:ip_diagonal_sensor}
    G = \Fc_\w \Diag(q)\Fc_\w^* + \Epsilon.
\end{equation}
Given a prior distribution $q\sim\Nc(q_0,\Cc_0)$, the form of the posterior distribution is an immediate extension of \cref{prop:posterior}.

\begin{proposition}\label{prop:posterior_sensor}
The design-dependent posterior distribution of the Bayesian inverse problem detailed above is $q \mid G\sim\Nc(q_\post(\w),\Cc_\post(\w))$, where
\begin{equation}\label{prop:posterior_sensor:eq:posterior_diagonal} 
\begin{split}
    q_\post(\w) & := \Cc_\post(\w) \left(
    \diag\left(
    \Fc_\w^*\Gmn^{-1}
        G
        \Gmn^{-1}\Fc_\w\right) + \Cc_0^{-1}q_0\right), \\
    \Cc_\post(\w) & :=\left( 
        \diag \circ \left[\left[\Fc^*\Gmnwi\Fc\right] 
            \kron
        \left[\Fc^*\Gmnwi\Fc\right]\right] \circ \Diag 
        +
        \Cc_0^{-1}
    \right)^{-1}
\end{split}
\end{equation}
and $\Gmnwi:=\Gmn^{-1/2}\Mw\Gmn^{-1/2}$.
\end{proposition}
Since both $\Gmn$ and $\Diag(\w)$ are diagonal and each $w_i\in \{0,1\}$, we can apply the same manipulation as in \cite[Sec.~2.4 and 3.1]{AttCon2022}, consistently substituting $\Diag(\w)\Gmn^{-1}\Diag(\w)$ with $\Gmn^{-1/2}\Diag(\w)\Gmn^{-1/2} = \Gmnwi$ everywhere. 
Note that $\Gmnwi$ is not invertible if entries $w_i$ of $\w$ are zero; as a consequence, we never use the inverse of $\Gmnwi$.

Let us now formulate the OED optimization problem and assume we aim to find optimal designs with $m_0\le m$ sensors. Assuming binary weights, this can be written using the $\ell_0$-``norm'' as
\begin{equation}\label{eq:ell_0 constraint}
\|\w\|_0:=\#\{k\in\N, \,k\leq m \mid \w_k\neq 0\}\le m_0.
\end{equation}
To circumvent combinatorial complexity, we relax \eqref{eq:ell_0 constraint} in two ways. First, we replace $\|\cdot\|_0$ by the $p$-norm $\|\cdot\|_p$ for small $p\ge 0$. Second, following earlier optimal design work, \cite{AlePetStaGha14,Ale21,ChowdharyAttiaAlexanderian2025}, we relax the requirement that the weight vector $\w$ have strictly binary entries $w_i\in \{0,1\}$, and instead permit continuous components $w_i\in [0,1]$. With this, we can formulate an auxiliary OED problem for passive imaging as
\begin{equation}\label{eq:OED}
\w^{*p,m_0}\in\mathop{\mathrm{argmin}}_{\w\in[0,1]^m,\, \|\w\|_p\leq m_0}\Jc(\w).
\end{equation}
Our ultimate goal is to determine $\w^{*0,m_0}$, a \emph{binary} optimal design, since it can be directly interpreted and implemented in practical sensor placement experiments. However, in the OED algorithm described in \cref{sec:OEDalg}, we instead compute a sequence of design vectors $\w^{*p,m_0}$ with $p>0$, because they are more tractable to obtain and can serve as approximations to binary designs. The objective $\Jc$ measures the quality of the posterior \eqref{prop:posterior_sensor:eq:posterior_diagonal}, through the A-optimal objective 
\begin{equation}\label{eq:A_optimal}
		\Jc(\w)  := \tr(\Cc_\post(\w)) = 
        \tr\left(
	\left(
        \diag\left(\Fc^*\Gmnwi\Fc 
            \Diag(\cdot)
        \Fc^*\Gmnwi\Fc\right) +
        \Cc_0^{-1}
    \right)^{-1}
    \right),
\end{equation}
which, by \cite{Mercer1909}, is equivalent to minimizing the average pointwise posterior variance.

The rest of this work focuses on theoretical and numerical aspects of A-optimal sensor placement for passive imaging, emphasizing its differences from the active imaging setting in \cite{Aarset2024,Ale21,AlePetStaGha14}, where OED was applied to the active (non-correlated data) inverse problem \eqref{eq:IP_uncorrelated_sensor}.
A key difference is that the design effect enters \eqref{eq:ip_diagonal_sensor} twice and therefore appears four times, rather than twice, in the posterior covariance. As a result, the convexity of $\Jc$ in $\w$ may be lost, and numerical experiments suggest that $\Jc$ may display mildly non-convex behavior. Therefore, global optimality guarantees no longer apply straightforwardly.
However, as \cref{thm:low_rank_OED} will show, $\Jc$ is monotone in every sensor weight, i.e., no sensor provides \enquote{negative information}, which greatly limits the impact of non-convexity on optimization. Moreover, numerical experiments suggest that local optima of $\Jc$ behave similarly to the global optima in \cite[Thm.~3]{Aarset2024}; see \cref{sec:Helmholtz}. To find optimal designs, we will require efficient evaluation of both the objective $\Jc$ and its gradient $\nabla\Jc$. Accordingly, we will next study how this can be done efficiently in a FEM-discretized setting.

\subsection{Parameter discretization}\label{ssec:FEM}
This section derives finite-dimensional matrix representations of the operators in $\Jc(\w)$ suitable for low-rank approximation. Its main results are \cref{lem:diag_form1} (diag and trace formulas) and \cref{lem:diag_form2} (a Schur-product representation capturing the role of the design).
We discretize both $f$ and its pointwise variance $q$, defined over $\Omega$ using the space $V_h^0\subset L^2(\Omega)$ of discontinuous, piecewise constant finite elements on the triangulation $\Omega_i$, $i=1,\ldots, n$ of $\Omega$. Piecewise constant basis functions are chosen to accommodate the low regularity of random sources and the desire to mimic the pointwise arguments in \eqref{subsec:uncorrelated}. The basis functions $\varphi_i$ are indicator functions of $\Omega_i$, and satisfy the important identity $\varphi_i\varphi_j=\delta_{ij}\varphi_i$.
Every $g_h\in V_h^0$ is identified with its coefficient vector $g:=(g_1,\ldots,g_n)^T\in \mathbb R^n$ via $g_h(\x) = \sum_{i=1}^n g_i \varphi_i(\x)$. Due to the discontinuity of $\varphi_i$, the mass matrix $M$ corresponding to $V_h^0$ is diagonal, $M=\diag(m_1,\ldots,m_n)$ with $m_i=\int_\Omega \varphi_i^2(\x)\,d\x = |\Omega_i|$, the area of the $i$-th element. We obtain the usual isomorphism between functions $f_h$, $g_h$ in $V_h^0$ equipped with the $L^2$-inner product and the coefficient vectors $\boldsymbol f$, $\boldsymbol g$ equipped with an $\M$-weighted inner product, i.e., $(\boldsymbol f,\boldsymbol g)_\M:=(\boldsymbol f, \M\boldsymbol g) = \boldsymbol f^T\M\boldsymbol g$. This implies that $(f_h,g_h)_{L^{2}}=(\boldsymbol f,\boldsymbol g)_\M$.
To emphasize when we use the $\M$-weighted inner product for the coefficients, we use the notation $\mathbb R_\M^n$.

Let us now consider a bounded linear operator $\mathcal A: V_h^0\to V_h^0$ and compute its representation in $\mathbb R_\M^n$ and its diagonal. One has
\begin{equation*}
    (\mathcal A f_h,g_h)_{L^2} = \sum_{i=1}^n\sum_{j=1}^n f_i g_j ( \mathcal A\varphi_i,  \varphi_j)_{L^2} = \boldsymbol f^T A^T \boldsymbol g = 
    \boldsymbol f^T A^T M^{-1} M\boldsymbol g =
    \left(M^{-1}A\boldsymbol f,\boldsymbol g\right)_M,
\end{equation*}
where the matrix $A$ has entries $A_{ij} = (\mathcal A\varphi_j,\varphi_i)_{L^2}$. Thus,
the coefficient vector for $\mathcal A f_h$ is $\M^{-1}A \boldsymbol f = \big(m_i^{-1} \sum_{j=1}^n A_{ij} f_j\big)_{i=1}^n$. Since $(f_h,\varphi_i)_{L^{2}}=m_if_i$ for any $i$, 
\begin{equation*}
    \mathcal A f_h = \sum_{i=1}^n (\M^{-1}A \boldsymbol f)_i \varphi_i =
    \sum_{i=1}^n\sum_{j=1}^n m_i^{-1}A_{ij}m_j^{-1}(f_h,\varphi_j)_{L^2}\varphi_i, 
\end{equation*}
that is, for any $\x\in\Omega$, $\mathcal A f_h(\x) = \int_\Omega a_h(\x,\y) f_h(\y)\d\y$, with the kernel function $a_h(\x,\y) := \sum_{i=1}^n \sum_{j=1}^n m_i^{-1}A_{ij}m_j^{-1} \varphi_i(\x) \varphi_j(\y)$. Note that the representation of a composition of operators is the product of their matrix representations. In particular, if $\mathcal A$ and $\mathcal B$ are represented by $M^{-1}B$ and $M^{-1}A$, respectively, then $\mathcal{AB}$ is represented by $M^{-1}AM^{-1}B$.
The diagonal and trace of an operator are extracted as follows.
\begin{lemma}\label{lem:diag_form1}
For a bounded linear $\mathcal A:V_h^0\to V_h^0$ with matrix representation $M^{-1}A$ and for a.e.~$\x\in\Omega$, one has
\begin{align*}
\operatorname{diag}(\mathcal A)(\x)  = 
\sum_{i=1}^n m_i^{-2} A_{ii} \varphi_i(\x) \quad \text{and} \quad
\tr\mathcal{A}  = \tr\left(M^{-1}A\right) = \sum_{i=1}^n m_i^{-1}A_{ii}.
\end{align*}
The first equality implies that the vector representation of $\diag(\Ac)$ is $M^{-2}\diag(A)$.
\end{lemma}
\begin{proof}
The finite rank nature of $\Ac$ implies that its spectrum contains only finitely many non-zero eigenvalues, allowing it to be embedded into $\HS(X,X^*)$ and to have a well-defined trace. The proof of the first equality now follows from the above derivation of $a_h$ and \cref{lemma:diag}. For the second identity, we note that $\big(m_i^{-1/2}\varphi_i\big)_{i=1}^n$ is an orthonormal basis for $V_h^0$, and thus
\begin{align*}
    \tr\mathcal{A} = \sum_{i=1}^n\left(\mathcal A [m_i^{-1/2}\varphi_i], m_i^{-1/2}\varphi_i\right)_{L^2} = 
    \sum_{i=1}^n m_i^{-1}A_{ii} = \tr\left(M^{-1}A\right).
\end{align*}
Equivalently, one can employ the well-known identity $\tr\mathcal{A}=\int_\Omega\diag(\mathcal{A})(\x)\d\x$.
\end{proof}
The preceding result relies crucially on the orthogonality of the basis functions $\varphi_i$ and on the relation $\varphi_i\varphi_j=\delta_{ij}\varphi_i$ satisfied by these functions. Note that deriving a comparable result for higher-order Lagrange basis functions would be considerably more involved.

\cref{lem:diag_form1} gives matrix-level formulas for the diagonal and the trace, which we will combine with Schur product structure in the next \cref{lem:diag_form2} to expose the design in the objective.
First, for $q_h=\sum_{i=1}^n q_i\varphi_i$, consider the operator $\Diag(q_h)\colon f_h\in V_h^0\mapsto q_hf_h\in V_h^0$; this operator is well defined, as 
$$
    q_hf_h = \bigg(\sum_{i=1}^n q_i\varphi_i\bigg)\bigg(\sum_{j=1}^n f_j\varphi_j\bigg) = \sum_{i=1}^n\sum_{j=1}^n q_if_j\varphi_i\varphi_j=\sum_{i=1}^nq_if_i\varphi_i,
$$
again owing to $\varphi_i\varphi_j=\delta_{ij}\varphi_i$. Accordingly, the matrix representation of $\operatorname{Diag}(q_h)$ is $\operatorname{Diag}(\boldsymbol q)\in \mathbb R^{n\times n}$.
We can thus consider products of operators as they occur in the posterior covariance in passive imaging.
\begin{lemma}\label{lem:diag_form2}
Let $\Ac$, $\Bc \colon V_h^0\to V_h^0$ be bounded and linear with matrix representations $M^{-1}A$, $M^{-1}B$, respectively. Then the linear mapping 
\begin{equation}\label{eq:q_h-map}
    q_h \in V_h^0 \mapsto \diag\left(
        \Ac\Diag\left(
            q_h
        \right)\Bc
    \right)\in V_h^0
\end{equation}
has the matrix representation
$$    
M^{-1}\left[(M^{-1/2}AM^{-1/2})\odot (M^{-1/2}B^TM^{-1/2})\right].
$$
\end{lemma}
\begin{proof}
    From the above discussion, it follows that the matrix representation of $\Ac\Diag\left(q_h\right)\Bc:V_h^0\to V_h^0$ is $\M^{-1}A\,\operatorname{Diag}(\boldsymbol q)\M^{-1}\B$. Using \cref{lem:diag_form1}, the function on the right hand side in \eqref{eq:q_h-map} has the coefficients
    \begin{equation*}
        \begin{split}
    \M^{-2}\diag(A\,\Diag(\boldsymbol q)\M^{-1}B) &=
    \M^{-1}\diag(\M^{-1}A\Diag(\boldsymbol q)\M^{-1}B) \\
    &= M^{-1}\diag(M^{-1/2}AM^{-1/2}\Diag(\boldsymbol q)M^{-1/2}BM^{-1/2}) \\
    &= M^{-1} \left[(M^{-1/2}AM^{-1/2})\odot (M^{-1/2}B^TM^{-1/2})\right]\boldsymbol q,              \end{split}
    \end{equation*}
In the second equality above, we use that both $M$ and $\Diag(\boldsymbol q)$ are diagonal matrices, and that for any matrix $C$ one has $\diag(M^{-1}C) = \diag(M^{-1/2}CM^{-1/2})$.
\end{proof}
This result mirrors the familiar identity $\diag \circ \left[{A}\kron{B}^T\right] \circ \Diag = {A}\schur{B}^T\in\R^{n\times n}$ for matrices ${A}, {B}\in\R^{n\times n}$, as stated in \cref{lemma:linear_algebra}, here adapted to involve mass matrices for finite element discretizations.

We also briefly discuss the discretization of the parameter-to-observable map $\Fc:X\to\R^m$ from \eqref{eq:IP_uncorrelated}. By restriction, we let $\Fc:V_h^0\to \R^m$, and note that for each $f_h\in V_h^0$ and each $i\in\N$, $i\leq m$,
\begin{equation}\label{eq:F_matrix}
    \left(\Fc f_h\right)_i = \sum_{j=1}^n \left(\Fc\varphi_j\right)_i f_j =  (F\boldsymbol f)_i,
\end{equation}
where the matrix representation $F\in\R^{m\times n}$ has entries $F_{ij} = \left(\Fc\varphi_j\right)_i$. From \cite[(3.4)]{BThGhaMarSta13}, the matrix representation of $\Fc^*:\R^m\to V_h^0$ is given as $M^{-1}F^T$.

\subsection{Design-dependent posterior discretization}\label{ssec:A-opt-discrete}
Combining \cref{prop:posterior_sensor} with \cref{ssec:FEM}, we now derive the matrix representation of the design-dependent posterior  covariance $\Cc_{\post}$. While we provide the exact representation below, the low-rank approximations needed for the efficient evaluation of the design objective and its gradient are developed in \cref{ssec:low-rank}.
\begin{proposition}\label{prop:posterior_matrix_representation}
Let $\Cc^{1/2}:V_h^0\to V^0_h$ have the matrix representation $M^{-1}C$, $C\in\R^{n\times n}$, and let $F\in\R^{m\times n}$ be the matrix representation of $\Fc:V_h^0\to\R^m$ as per \eqref{eq:F_matrix}. Given $\w\in\{0,1\}^m$, the posterior covariance operator $\Cc_\post(\w):V^0_h\to V^0_h$ in \eqref{prop:posterior_sensor:eq:posterior_diagonal} has the matrix representation
\begin{equation}\label{prop:posterior_matrix_representation:eq:representation}
\begin{gathered}
   M^{-1}CM^{-1/2}\left(
        M^{-1/2}CM^{-1}\left[\hat{F}_\w
            \schur
        \hat{F}_\w\right]M^{-1}CM^{-1/2}
    + I\right)^{-1}M^{-1/2}C, \\ 
     \text{where}\quad \hat{F}_\w :=
    M^{-1/2}F^T\Gmnwi FM^{-1/2},
\end{gathered}
\end{equation}
recalling $\Gmnwi = \Gmn^{-1/2}\Mw\Gmn^{-1/2}$.
\end{proposition}

\begin{proof}
As $M^{-1}F^T$ is the matrix representation of $\Fc^*:\R^m\to V^0_h$, the matrix representation of $\widehat\Fc_\w$ is $M^{-1}F^T\Gmnwi F=M^{-1}\big[M^{1/2}\widehat F_\w M^{1/2}\big]$. It follows from \cref{lem:diag_form2} that the matrix representation of $\diag \circ \big[\widehat\Fc_\w \kron \widehat\Fc_\w\big] \circ \Diag$ is $M^{-1}\big[\hat{F}_\w\schur\hat{F}_\w\big]$. Moreover, the elementary identity 
\[
    (\Ac + \Cc_0^{-1})^{-1} = \Cc_0^{1/2}\left(\Cc_0^{1/2}\Ac\Cc_0^{1/2} + I\right)^{-1}\Cc_0^{1/2}
\]
for arbitrary $\Ac:V^0_h\to V^0_h$ permits us to represent $\Cc_\post(\w)$ as
\[
    M^{-1}C\left(
        M^{-1}CM^{-1}\left[\hat{F}_\w
            \schur
        \hat{F}_\w\right]M^{-1}C
    + I\right)^{-1}M^{-1}C.
\]
Finally, the identity 
$
    (M^{-1}A + I)^{-1}=M^{-1/2}(M^{-1/2}AM^{-1/2}+I)^{-1}M^{1/2}
$
for arbitrary $A\in\R^{n\times n}$ yields the claimed form of the matrix representation.
\end{proof}
The representation \eqref{prop:posterior_matrix_representation:eq:representation} 
is unsuitable for repeated evaluations since the dense matrix $\hat{F}_\w\schur\hat{F}_\w\in \mathbb R^{n\times n}$, which depends on the discretized forward operator $F$, has to be reassembled for each new design. Moreover, the combination of prior preconditioning with $C$ and the required matrix inversion makes repeated evaluation of \eqref{prop:posterior_matrix_representation:eq:representation} impractical in practice.
To reduce the computational cost and remove its dependence on the discretization dimension $n$ during the online phase, we leverage the observation that some components of \eqref{prop:posterior_matrix_representation:eq:representation} do not depend on the design $\w$ and can thus be precomputed. Furthermore, some components allow for efficient low-rank approximations, as explained below.

\subsection{Two-level low rank approximation of posterior update}\label{ssec:low-rank}
We use two nested low-rank decompositions, both independent of the design. The first, in \eqref{eq:QR1}, targets the forward map and uses an $\ell$-dimensional space to approximate the range of the map. This yields the identity 
$\hat{F}_\w = QR\diag(\w)R^TQ^T$.
A single compression that does not involve the prior is, however, insufficient. Thus, using a second compression in \eqref{eq:QR2}, we extract the dominant $\hat\ell$-dimensional subspace of a product space operator that involves the prior and  the Khatri-Rao product $Q^T\khra Q$. The thin QR decompositions of these matrices as follows:
\begin{align} 
    QR & = M^{-1/2}F^T\Gmn^{-1/2} \in\R^{n\times m}, \label{eq:QR1}\\
    \Qh\Rh & = M^{-1/2}CM^{-1}\left[Q^T\khra Q\right]^T \in\R^{n\times \ell^2}, \label{eq:QR2}
\end{align}
with $Q\in\R^{n\times\ell}$, $R\in\R^{\ell\times m}$ and $Q^TQ=I_\ell\in\R^{\ell\times\ell}$, as well as
$\Qh\in\R^{n\times\ellh}$, $\Rh\in\R^{\ellh\times\ell^2}$ and $\Qh^T\Qh=I_\ellh\in\R^{\ellh\times\ellh}$, $\ell$, $\ellh\in\N$, $\ell\leq\min\{m,n\}$, $\ellh\leq\min\{\ell^2,n\}$.

Because the right-hand matrices in \eqref{eq:QR1} and \eqref{eq:QR2} are independent of the design $\w$, these factorizations can be computed in advance and stored; see \cite[App.~A.1]{Aarset2024} for an analogous active-imaging construction. The payoff, detailed in \cref{thm:low_rank_OED} below, is that the online evaluation of $\Jc(\w)$ and its gradient reduces to linear algebra on precomputed matrices of size $\ell,\ellh$, independent of further applications of the forward or prior maps.
We now introduce the following notation for $\w \in [0,1]^m$:
\begin{align}
    \widehat C & :=\Qh^T \M^{-1/2}C\M^{-1}CM^{-1/2} \Qh\in\R^{\ellh\times\ellh}, \label{eq:Chat}\\
    L_\w & := R\Mw R^T\in\R^{\ell\times\ell}, \label{eq:Lw}\\
    \Lh_\w & := \Rh\left[L_\w\kron L_\w\right]\Rh^T\in\R^{\ellh\times\ellh}, \label{eq:Lhatw}\\
    \Lt_\w(i) & := \mat\left(
        \left[\Rh^T
        \left(\Lh_\w  + I_{\ellh}\right)^{-1}
        \widehat D^T\right]_{:i}
    \right) \in \R^{\ell\times\ell} \qquad \text{for $i\in\N$, $i\leq\ellh$,} \label{eq:Ltildew}
\end{align}
where $\widehat C$ has the Cholesky decomposition $\widehat C = \widehat D^T \widehat D$, $\widehat D\in\R^{\ellh\times \ellh}$. Each of these objects plays a distinct role in what follows. The matrix $\widehat C$ is a design-independent compression of the (square of the) prior covariance, projected onto the second low-rank subspace. The matrix $L_\w$ encodes the design into the first compression as needed in $\hat F_\w$. The matrix $\Lt_\w$ lifts $L_\w$ into the second compression, reflecting the quadratic dependence of the correlation data on the design. Finally, the matrices $\Lt_\w(i)$ only appear in the gradient. Note that we consistently use the hat symbol $\widehat\cdot$ to denote matrices belonging to the second low-rank decomposition, which have either their input or output in $\R^{\ellh}$.
Using the matrices \eqref{eq:QR1}--\eqref{eq:Ltildew}, the next two theorems present low-rank approximations of the design-dependent posterior and the OED objective function.

\begin{theorem}[Low-rank posterior using QR]\label{thm:low_rank_posterior}
Using the QR-decompositions \eqref{eq:QR1} and \eqref{eq:QR2}, the matrix representation \eqref{prop:posterior_matrix_representation:eq:representation} of $\Cc_\post(\w)$ can be written as
\begin{equation}\label{eq:low_rank_OED:posterior:covariance}
C_\post(\w) := M^{-1}C\M^{-1/2}\Big(
I - \Qh \Qh^T + 
\Qh \big(\Lh_\w  + I_{\ellh}\big)^{-1}\Qh^T
\Big)\M^{-1/2}C.
\end{equation}
Moreover, for covariance matrix data $G\sim\Nc(\Fc\Diag(q)\Fc^*,\Gmn\otimes\Gmn)$, the coefficient vector $q_\post(\w)$ of the posterior mean is
\begin{equation}\label{eq:low_rank_OED:posterior:mean}
\begin{aligned}
\q_\post(\w) & = C_\post(\w)\Big(
    M^{-1}\diag\bigl(
    QR\Gmn^{-1/2}\Mw G\Mw \Gmn^{-1/2}R^TQ^T
    \bigr)\\ &
    + 
    C^{-1}MC^{-1}M\q_0\Big),
\end{aligned}
\end{equation}
where $\q_0$ is the coefficient vector of the prior mean $q_0$.
\end{theorem}

\begin{proof}
The proof follows the same strategy as \cite[Thm.~8]{Aarset2024}, with minor adjustments required to account for the correlation data. To derive the low-rank form of the posterior covariance's matrix representation from \eqref{prop:posterior_matrix_representation:eq:representation}, we note that $\hat{F}_\w =
    QR\Mw R^TQ^T = QL_\w Q^T$. Using \cref{lem:schur_decomp} shows that 
\[
\hat{F}_\w \schur \hat{F}_\w = \left[Q^T\khra Q\right]^T 
    \left[
    	L_\w \kron L_\w
    \right]
    \left[Q^T\khra Q\right],
\]
and using \eqref{eq:QR2}, we obtain
\begin{align*}
	& \hphantom{=} M^{-1}CM^{-1/2}\left(
        M^{-1/2}CM^{-1}\left[\hat{F}_\w
            \schur
        \hat{F}_\w\right]M^{-1}CM^{-1/2}
    + I\right)^{-1}M^{-1/2}C \\
	& = 
	M^{-1}CM^{-1/2}\left(
	\Qh\Rh
	\left[
    	L_\w \kron L_\w
    \right]
    \Rh^T\Qh^T + I
	\right)^{-1}M^{-1/2}C.
\end{align*}
Using a Woodbury identity argument (e.g., \cite[p.~7 (23)]{Henderson1981}) leads to 
\[
\left(
\Qh\Rh
	\left[
    	L_\w \kron L_\w
    \right]
    \Rh^T\Qh^T + I
\right)^{-1} = I - \Qh\Qh^T + \Qh \left(\Lh_\w  + I_{\ellh}\right)^{-1}\Qh^T
\]
which proves \eqref{eq:low_rank_OED:posterior:covariance}.

The posterior mean coefficient vector \eqref{eq:low_rank_OED:posterior:mean} is obtained from the structure of the posterior and its expansion in terms of the first decomposition; observe that the final term represents the noise correction of the sample covariance matrix by subtracting $\Gmn$, as discussed following \cref{corr:wishart_inverse}.
\end{proof}
The posterior mean \eqref{eq:low_rank_OED:posterior:mean} is formulated in terms of the covariance matrix estimator $G$. Since $G$ is generally dense, this can become computationally infeasible for large sensor arrays of size $m$. Owing to \cref{prop:adjoint}, however, the necessary backpropagation $\diag(\Fc^*G\Fc)$ can alternatively be implemented using data samples $(\g^{(i)})_{i=1}^N$ from \eqref{eq:IP_uncorrelated_sensor}, $\g^{(i)}\sim\Nc(0,\Fc_\w\Qc\Fc^*_\w + \Gmn)$. For large $N$, \cref{lemma:wishart_properties} indicates that $\tfrac{1}{N}\sum_{i=1}^N\g^{(i)}\kron\g^{(i)} - \Gmn$ serves as a suitable approximation to $G\sim\Nc(0,\Gmn\kron\Gmn)$. Expressing everything in terms of the $\g^{(i)}$ yields the following alternative representation
\begin{equation}\label{eq:low_rank_OED:posterior:mean:samples}
\begin{aligned}
\q_\post(\w) & \approx  
    C_\post(\w) M^{-1}\bigg(
    \Big[\sum_{i=1}^N \left(QR\Gmn^{-1/2}\Mw\g^{(i)}\right)^2\Big]  \\
    & -
    M^{-1}\diag\left(
    QR\Mw R^TQ^T
    \right)
    - 
    C^{-1}MC^{-1}M\q_0
    \bigg).
\end{aligned}
\end{equation}
In contrast to \eqref{eq:low_rank_OED:posterior:mean}, this expression can be computed incrementally as additional samples $\g^{(i)}$ are obtained, after which each sample may be discarded; at no point is an $m\times m$ matrix formed in memory. The new $\diag$ term arises from backpropagating the noise covariance $\Gmn$ and generally cannot be excluded; it can alternatively be computed as $[QR]\schur[QR]\w = [Q^T\khra Q]^T[I\kron I][R\khra R^T]\w$; see \cref{lemma:linear_algebra} and \cref{lem:schur_decomp}.

\begin{theorem}[A-optimal objective via QR]\label{thm:low_rank_OED}
In a neighborhood of $[0,1]^m$, the objective $\Jc$ and its gradient satisfy
\begin{equation}\label{eq:low_rank_OED:objective}
\begin{aligned}
	\Jc(\w) & = \tr(M^{-1}CM^{-1}C) - \tr(\Ch) + \tr\Big(
		 \left(\Lh_\w  + I_{\ellh}\right)^{-1}\Ch\Big), \\
	\nabla\Jc(\w) & = 
	-\operatorname{SumCol}\bigg(
        R \schur \Big[
            \Big(
                \sum_{i=1}^\ellh \Lt_\w(i)L_\w \Lt_\w(i)^T+\Lt_\w(i)^TL_\w\Lt_\w(i)
            \Big)R
        \Big] 
    \bigg),
\end{aligned}
\end{equation}
where $\operatorname{SumCol}$  maps a matrix to the vector containing the summed entries of each column. Moreover, the function $\Jc$ is monotone decreasing in $\w$, meaning that $\Jc(\w)\leq\Jc(\w')$ holds whenever $\w\leq\w'$ componentwise.

\end{theorem}
\begin{proof}
Using the cyclic nature of the trace in \eqref{eq:low_rank_OED:posterior:covariance} yields that
\[
\Jc(\w) =  \tr\bigg(\Big( I - \Qh\Qh^T + \Qh \left(\Lh_\w  + I_{\ellh}\right)^{-1}\Qh^T \Big) \M^{-1/2}C\M^{-1}C\M^{-1/2}\bigg).
\] 
The two first terms of $\Jc(\w)$ follow from splitting out the bracket and again employing cyclicity of the trace. We proceed to derive the claimed formula for the gradient. Fix any $k\leq m$ and differentiate, invoking, for example, \cite[Thms.~B.17 \& B.19]{Ucinski2004}:
\begin{equation}\label{eq:low_rank_OED:gradient_computation}
\begin{aligned}
\frac{\partial}{\partial w_k}\Jc(\w) & =
\frac{\partial}{\partial w_k}\tr\left( \left(\Lh_\w  + I_{\ellh}\right)^{-1}\Ch\right) \\
& = -\tr\Big(\widehat D
		 \left(\Lh_\w  + I_{\ellh}\right)^{-1}\Rh
		\left[L_\w\kron L_{\e_k}\right]
		\Rh^T 
        \left(\Lh_\w  + I_{\ellh}\right)^{-1}\widehat D^T \\	
    & \hspace{6.1ex} + \widehat D
		 \left(\Lh_\w  + I_{\ellh}\right)^{-1}\Rh
		\left[L_{\e_k}\kron L_\w\right]
		\Rh^T 
      \left(\Lh_\w  + I_{\ellh}\right)^{-1}\widehat D^T
	\Big)
\end{aligned}
\end{equation}
using the quotient and chain rules, as $\w\mapsto L_\w$ is linear in $\w$ and the trace is linear. Applying \cref{lem:Kronecker_trace} yields
\begin{equation}\label{eq:low_rank_OED:gradient_computation_II}
\begin{gathered}
\frac{\partial}{\partial w_k}\Jc(\w) 
= 
-\sum_{i=1}^\ellh
\tr\left(
	L_{\e_k}\Lt_\w(i)^TL_\w\Lt_\w(i) +
    L_{\e_k}\Lt_\w(i)L_\w\Lt_\w(i)^T
\right) \\
= -
\tr\bigg(
	\bigg[
    \sum_{i=1}^\ellh\Lt_\w(i)^TL_\w\Lt_\w(i) +
    \Lt_\w(i)L_\w\Lt_\w(i)^T
    \bigg]R\Diag(\e_k)R^T
\bigg).
\end{gathered}
\end{equation}
To simplify the above expression, we note that $R\Diag(\e_k)R^T=R_{:,k}R_{:,k}^T$, i.e., it is a rank-$1$ matrix, and thus, for arbitrary compatible matrices $A$, satisfies 
\[
    \tr(A R\Diag(\e_k)R^T) = \tr\left(R_{:,k}^TA R_{:,k}\right) = R_{:,k}^TAR_{:,k} = (\mathrm{SumCol}(R\schur[AR]))_k
\]
by cyclicity and with the last equality following by a simple index chase. The claimed identity is now immediate. 
Finally, monotonicity can be seen by noting in either \eqref{eq:low_rank_OED:gradient_computation} or \eqref{eq:low_rank_OED:gradient_computation_II} that the trace is taken over products of positive semidefinite matrices, i.e.~the trace is non-negative and so the gradient is non-positive.
\end{proof}
As mentioned above, \cref{thm:low_rank_posterior,thm:low_rank_OED} are valid only under the assumption that each sensor produces a single scalar observation. This assumption is already violated when working with complex-valued data, which we naturally regard as two separate real-valued measurements per sensor. The situation of uncorrelated complex observations is addressed in \cite[Thm.~17]{Aarset2024}, and a corresponding generalization to time-series data is given in \cite{AarsetNguyen2025}. Moreover, in light of \eqref{prop:adjoint:eq:L1diag}, extending the framework to handle multiple measurements in the passive imaging context is straightforward.

\subsection{Computational complexity}\label{ssec:complexities}
Next, we summarize the computational complexity of the proposed method.
We divide the required computations into an offline and an online phase. During the offline phase, design-independent quantities are precomputed and stored. In the online phase, an iterative algorithm (described in \cref{sec:OEDalg}) is used to compute optimal designs. This algorithm entails repeatedly evaluating the optimal design objective and its gradient for given design vectors.
In most applications, including the numerical examples in \cref{sec:Helmholtz}, the applications of $F$ or its adjoint are substantially more expensive than the applications of the prior covariance. Thus, we discuss the computational complexity of the offline phase in terms of these operations.
For convenience, we recall the following dimensions: The discretization size of the source domain is $n$, the number of measurement points is $m$, $\ell$ is the dimension of the low-rank factorization of the forward map, and $\ellh$ is the low-rank dimension of the product space due to the correlation data. Note that since we are interested in the scalability of our computational methods, we focus on large discretization dimensions $n$. In the offline phase below, we neglect operations whose complexity does not scale with $n$ and operations that scale linearly with $n$ but have a small constant (such as multiplications with the diagonal mass matrix $M$).

\paragraph{Offline phase} 
In this phase, the computational cost to compute the two low-rank decompositions following \cref{thm:low_rank_OED} is dominated by forward and adjoint PDE solves to construct the low-rank approximations \eqref{eq:QR1}.
\begin{enumerate}[leftmargin=15pt]
    \item Computing a factorization of $F^T \Gmn^{-1} \in \R^{n\times m}$—for example, via a randomized subspace iteration algorithm \cite[Alg.~4.4]{HalMarTro11}—requires $t_F(2q+1)(\ell+p)$ operations involving $F$. In this expression, $p$ is the oversampling parameter, while $q$ is a small integer (which in some algorithms can even be chosen as $0$). Note that, in contrast to earlier work (e.g., \cite{AlePetStaGha14}), the operator $F^T\Gmn^{-1}\in\R^{n\times m}$ does not incorporate the prior covariance, which is typically crucial for enabling efficient low-rank approximations. In our setting, finite-rank approximations are instead governed exclusively by the finite set of observations and the properties of $F$. This difference stems from working with correlation data. The prior covariance appears only later, in a subsequent, second-level approximation that we describe next.
    \item The low-rank approximation of $M^{-1}C\M^{-1/2}[Q^T\khra Q]^T\in\R^{n\times\ell^2}$ now does not require additional applications of $F$ since it uses the orthogonal matrix $Q$ from the above step. Using a similar algorithm as above has complexity $t_C(2q+1)(\ellh+p)$ for the applications of the prior covariance factor $C$, and, more importantly, complexity $O(n^2\ell(\ellh+p))$ for the applications of $Q^T\khra Q$, as discussed in \cref{corr:linear_algebra_complexities}.
\end{enumerate} 
Remaining offline steps include the assembly of $\Ch$ defined in \eqref{eq:Chat}, and computing $\tr(\Cc_0)$ and $\tr(\Ch)$, which has negligible complexity compared to the two steps detailed above.

\paragraph{Online phase} 
The OED algorithm detailed in the next section requires repeated evaluation of the objective $\Jc(\w)$ and its gradient $\nabla\Jc(\w)$ for given $\w\in\R^m$. Thus, we next study the computational complexity of these operations, which amounts to computing \eqref{eq:Lw}, \eqref{eq:Lhatw} and \eqref{eq:Ltildew}, followed by \eqref{eq:low_rank_OED:objective}. Due to the computations in the offline phase, all operations in the online phase are independent of the discretization dimension $n$, but only depend on the low-rank dimensions $\ell$ and $\hat\ell$. The dominating computations are listed below.   

\begin{enumerate}[leftmargin=15pt]
\item Assembly of $L_\w\in R^{\ell\times\ell} $ defined in \eqref{eq:Lw} requires $\mathcal O(\ell^2m_\#)$ operations, where $m_\#:=\#\{k\in\N, \,k\leq m\mid\w_k\neq 0\}\leq m$.
\item Assembly of $\Lh_\w \in R^{\ellh\times\ellh}$ from \eqref{eq:Lhatw} requires ${\mathcal O}(\ellh\ell ^3+\ellh ^2\ell^2)$ operations (see \cref{corr:linear_algebra_complexities}), and computing a Cholesky factorization of $\Lh_\w  + I_{\ellh}$ has a complexity of ${\mathcal O}(\ellh^3)$.
\item Evaluation of $\Jc(\w)$ requires the computation of $\Lh_\w^{-1}\Ch\in\R^{\ellh\times\ellh}$, which has a complexity of
$\mathcal O(\ellh^3)$, followed by the evaluation of the trace.
\item Computation of the gradient $\nabla\Jc(\w)$ requires computation of the $\ellh$ matrices $\Lt_\w(i)\in\R^{\ell\times\ell}$ defined in \eqref{eq:Ltildew}, which is of ${\mathcal O}(\ell^2\ellh^2 + \ellh^3)$ complexity. The remaining operations required to compute the gradient are of lower complexity.

\end{enumerate}
Note that since $\ellh\le \ell^2$, the operational complexity of evaluation of the objective and its gradient is ${\cal O}(
\ellh^2\ell^2)$. 

\begin{table}[tb]
\centering
\caption{Summary of the leading computations required for the offline and online phases in terms of parameter-to-observable applications $F,F^T$, applications of prior covariance factor $C$, and linear algebra floating point computations. Here, $n$ denotes the discretization dimension of the source parameter, $\ell$ and $\hat{\ell}$ are the dimensions of the low-rank approximations of the parameter-to-observable map $F$ and its outer product arising due to correlations, respectively, and $q$ and $p$ are small integer values corresponding to power iterations and oversampling; see \cref{ssec:complexities} for details.} \label{tab:complexity}
\begin{tabular}{c|c|c|c}
      & \# forw/adjoint PDEs & \# prior covariance  & {$\mathcal O$}(flops) \\ \hline
\rule{0pt}{1.1\normalbaselineskip}offline & $(2q+1)(\ell+p)$  & $(2q+1)(\ellh+p)$  &  $n^2\ell\ellh$ \\ \hline
\rule{0pt}{1.0\normalbaselineskip} \makecell{online per eval.\\ of $\Jc$ and $\nabla\Jc$}  & ---  & --- & $\ell^2(\ellh^2 + m_\#)$ \\ 
\end{tabular}
\end{table}

\paragraph{Complexity summary}
Table \ref{tab:complexity} provides an overview of the dominant computational tasks in the offline and online stages.
The primary motivation for using the two-level offline compression \eqref{eq:QR1} and \eqref{eq:QR2} is to lower the number of forward/adjoint PDE solves from order $\ellh$ (which can be as large as $\sim\ell^2$) down to order $\ell$. Notably, the offline phase includes an ${\cal O}(n^2)$ term, stemming from the Khatri-Rao product in \eqref{eq:QR2}. This term, which may dominate the overall computation, is a consequence of the structure of passive imaging due to its use of correlation data.
As previously noted, the computations in the online phase do not depend on the discretization dimension $n$. Instead, they rely solely on the low-rank dimensions $\ell,\ellh$, which are determined by the physical properties of the underlying continuous inverse problem.

\section{Optimal design algorithm}\label{sec:OEDalg}
In what follows, we fix a target number $m_0\in\N$, $m_0\leq m$ of sensors and let $\w^*:=\w^{*1,m_0}$ denote a (fixed) $1$-relaxed non-binary optimal design solving \eqref{eq:OED} with $p=1$. Building on
\cref{thm:low_rank_OED}, we can now devise an efficient algorithm to solve \eqref{eq:OED}.
We employ the \emph{redundant-dominant $p$-continuation algorithm} introduced in \cite{Aarset2024} to obtain approximately optimal configurations for binary sensor placement. Note that one of the main results of the above article---its global optimality analysis---cannot currently be extended to the passive imaging framework because the necessary convexity conditions are not satisfied. Thus, the redundant-dominant algorithm cannot guarantee finite-step convergence for any value of $p$. In particular, while in \cite{Aarset2024} the initialization $\w^{*}$ of Algorithm \ref{alg:p_cont}, obtained by solving \eqref{eq:OED} with $p=1$, was proven to be a \emph{globally optimal non-binary design}---and thus yielded a lower bound on the objective value achievable by \emph{any} design, whether binary or not---this property no longer holds in the current setting. Here, we can at most expect that this initialization corresponds to a good local minimum. However, as will be shown in \cref{sec:Helmholtz}, the algorithm performs well in practice: a comparable bounding effect can be observed empirically, and for larger values of $m_0$, the resulting outputs $\wpcontnom$ outperform all competing binary designs considered. 

\begin{algorithm}[tb]
\caption{Binary OED by $p$-continuation via redundant-dominant classification}\label{alg:p_cont}
\begin{algorithmic}[1]
\Require Initialization as non-binary design $\w^{(0)}:=\w^{*}$ by approximately solving \eqref{eq:OED} for $p=1$ (e.g.~via the Sequential Least-Squares Quadratic Programming (SLSQP) algorithm \cite{Kraft1988}), set $i=0$ and continuation parameter $\delta\in(0,1)$.
\While {$\w^{(i)}$ has entries significantly different from $0$ and $1$}
	\State $p \gets (1-\delta)p$ and $i\gets i+1$.
    \State Construct $\Jc^p(\z) := \Jc(\z^{1/p})$.
    \State Compute $\nabla \Jc^p(\z) = \frac{1}{p}\nabla\Jc(\z^{1/p})\z^{1/p-1}$
	\State Solve the non-convex constrained problem (e.g.~via the SLSQP algorithm)
	\begin{equation} \label{eq:p_cont}
	\argmin_{\z\in[0,1]^m, \, \sum_{k=1}^mz_k\leq m_0}\Jc^p(\z),
	\end{equation}
	initialized at $\z := (\w^{(i-1)})^p$ and keeping dominant and redundant indices ($w^*_k=1$, resp.~$w^*_k=0$) fixed, returning $\z^{(i)}$.
	\State $\w^{(i)} \gets (\z^{(i)})^{1/p}$.
\EndWhile
\Ensure Binary design $\wpcontnom:=\w^{(i)}$ as approximate solution of \eqref{eq:OED}.
\end{algorithmic}
\end{algorithm}

\section{A numerical study for the  Helmholtz equation}\label{sec:Helmholtz}

In this section, we study our optimal design algorithms for the passive imaging problem \eqref{eq:ip_diagonal_sensor} using \cref{alg:p_cont}. We also demonstrate how the designs reduce the pointwise variance, show MAP estimates of the resulting posterior distributions, compare optimal and random designs, and study qualitative differences between the designs found for passive imaging problem in comparison to those found for the corresponding inverse problem \eqref{eq:IP_uncorrelated_sensor} with active, i.e., non-correlation data.

\subsection{Forward and inverse problem setup}\label{sec:Helmholtz:forward}
As the governing equation, we use the Helmholtz equation on a circular domain $B_1 \subset \mathbb R^2$ of unit radius containing three rectangular, sound-hard scatterers $S_1$, $S_2$, and $S_3$, as illustrated in \cref{fig:domain}. The real-valued source term $f$ is confined to the subdomain $\Omega \subset B_1$, defined as the disk of radius $0.35$. On the outer boundary of $B_1$ we impose impedance boundary conditions, while the scatterers satisfy homogeneous Neumann conditions.
Thus, the governing equation for complex-valued $u$ is
\begin{equation}\label{eq:Helmholtz}
\begin{split}
    -\Delta u - k^2 u & =  \mathcal E f \quad\: \text{ in } B_1\\
    \frac{\partial u}{\partial n} & =
    \begin{cases}
    iku \text{ on } \partial B_1,\\
    0 \quad\text{ on } \partial (S_1\cup S_2\cup S_3),
    \end{cases}
\end{split}
\end{equation}
where $k>0$ and $\mathcal E \colon L^2(\Omega)\to L^2(B_1)$ denotes the extension-by-zero operator. The sources $f$ are samples from $\mathcal N(0,\Diag(q))$, where $q\in L^\infty(\Omega)$, $q\ge 0$ denotes the pointwise variance function. In \cref{fig:domain}, the left panel displays the $344$ candidate sensor locations in $B_1\setminus\Omega$ together with the truth covariance function $q$ we seek to infer. We set $k=50$ and, in \cref{fig:domain}, show a Helmholtz solution $u$ of \eqref{eq:Helmholtz} generated with a random source $f\sim\Nc(0,\Diag(q))$. Since the wavelength is $(2\pi)/k\approx 1/8$, this corresponds to about $16$ wavelengths across $B_1$. The parameter-to-observation map $\Fc$ evaluates the Helmholtz solution $u$ at the sensor locations, splitting real and imaginary measurements as per the discussion in \cref{sec:passive_imaging}, and the correlated parameter-to-observation map $\Fch \colon L^\infty\to \mathbb R^{m\times m}$ is defined as in \eqref{eq:ip_diagonal_sensor}.

\begin{figure}
\centering
\begin{tikzpicture}
\node at (1,4.3) {\includegraphics[width=0.3\linewidth]{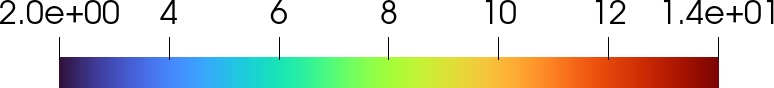}};
\node at (1,1) {\includegraphics[width=0.44\linewidth]{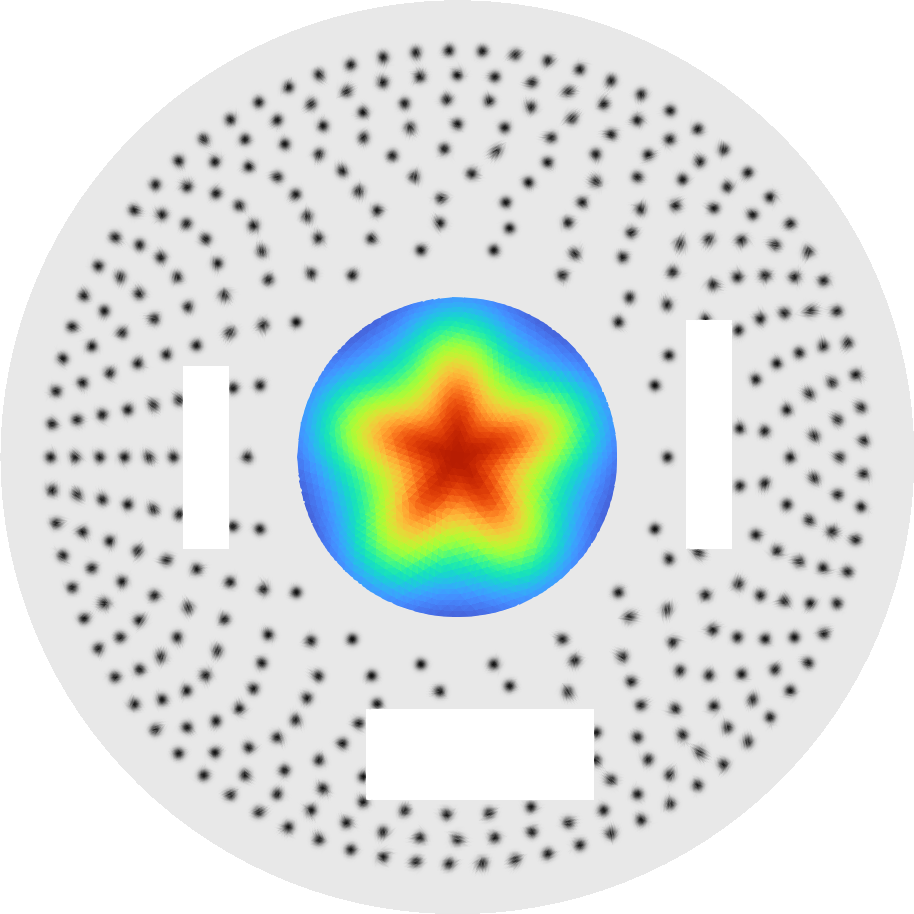}};
\node at (7.3,4.3) {\includegraphics[width=0.3\linewidth]{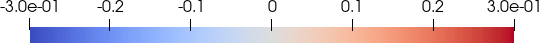}};
\node at (7.4,1) {\includegraphics[width=0.44\linewidth]{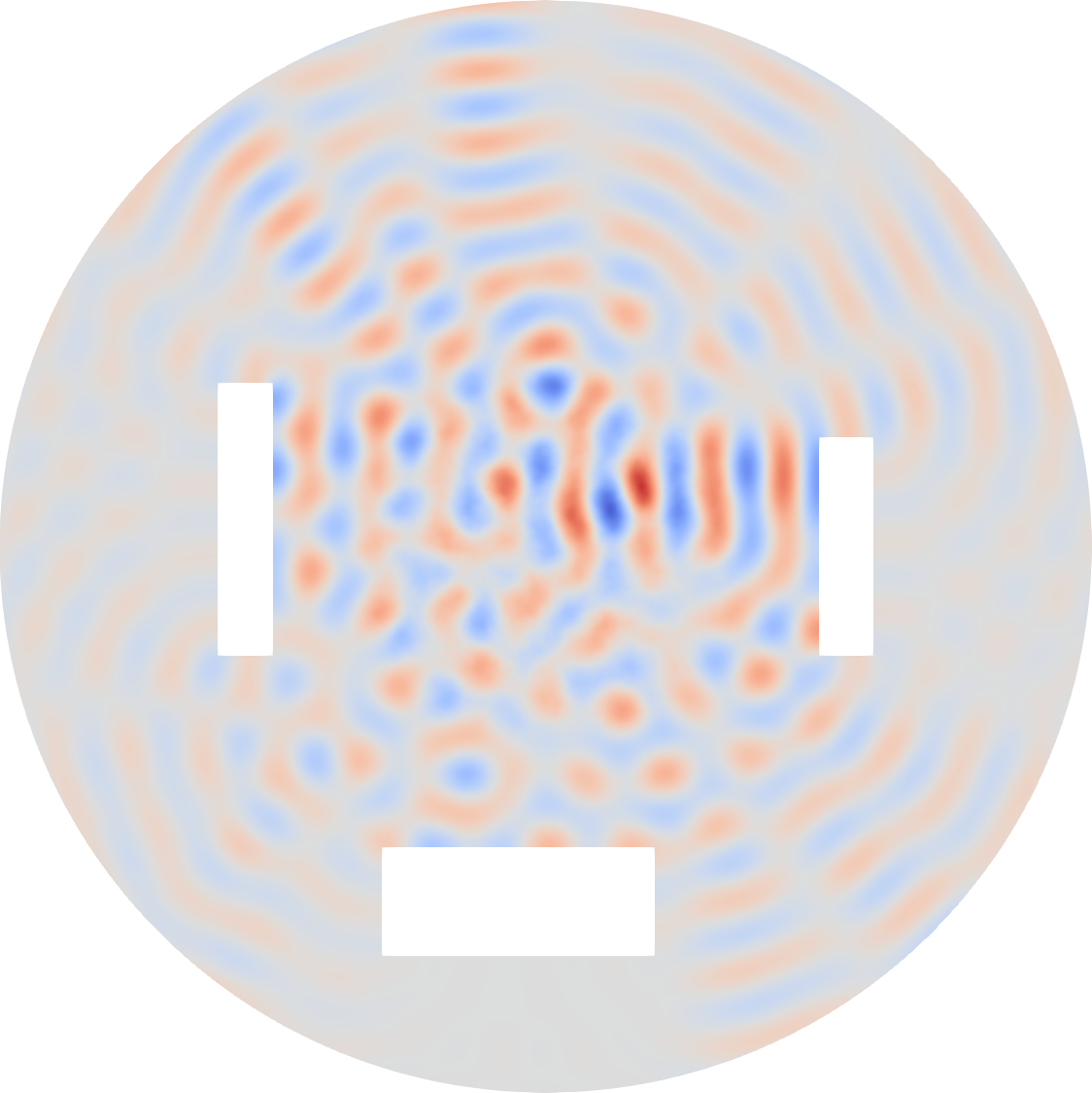}};
\end{tikzpicture}
\caption{\emph{Left:} Domain $\Omega$ with truth source covariance function $q$ (inner circle). Shown as black dots are the sensor candidate locations. \emph{Right:} Wave field resulting from a random source $f\sim\Nc(0,\Diag(q))$ in the inner circle $\Omega$. }\label{fig:domain}
\end{figure}

We next define the prior and noise distributions for the Bayesian inverse problems.
As prior for $q$, we use the Gaussian distribution with covariance
 operator $\Cc_0:=(\alpha\Delta+I)^{-2}$ with $\alpha={0.02}$,
 resulting in $\Cc_0^{1/2}=(\alpha\Delta+I)^{-1}$. Here, the Laplace
 operator uses Robin boundary conditions on $\partial\Omega$ with the
 empirical coefficient $\beta=\sqrt{\alpha}/1.42$ to mitigate
 undesired boundary effects in the covariance operator, see
 \cite{Daon2018}. To ensure that the prior mean is in the
 Cameron-Martin space for this covariance operator, it is chosen as
 $q_0:=\Cc_0^{1/2}c$, i.e., the half-power of the prior covariance
 applied to the constant function with value $c$, where $c$ is chosen
 such that $\int_{\Omega}q_0(\x)\d \x=6$. This results in
a close-to-constant $q_0$ whose value is approximately three times the
 average pointwise standard deviation associated with $\Cc_0$.
For the noise covariance, we use $\Gmn=\gamma^2\rm{I}$, with $\gamma>0$ given as $1\%$ of the average correlation data magnitude as generated by the prior, amounting to $\gamma\approx 0.0322$. To compute this average, we draw prior samples $q\sim\Nc(q_0,\Cc_0)$ and evaluate the correlation data forward model at samples of the corresponding noise realizations $f\sim\Nc(0,\Diag([q]_+))$, where $[\cdot]_+:=\max(\cdot,0)$.

\subsection{Discretization and computational setup}\label{ssec:Helmholtz-approx}
For the computations in this section, we build on the \textsc{NGSolve} python package \cite{NGsolve}. To discretize the Helmholtz equation, we use quadratic finite elements on an unstructured triangular grid consisting of $8012$ triangles, corresponding to a total of $16220$ unknowns. The right-hand side source, which is a realization of white noise, is discretized in $\Omega$ using piecewise constant basis functions, resulting in $n=4736$ unknowns. The mesh within $\Omega$ is refined to obtain an accurate representation of this source. To ensure proper white noise scaling for $f$, we multiply the piecewise constant function by the inverse square root of the mass matrix \cite{BThGhaMarSta13}.

As discussed in \cref{sec:OED}, a randomized subspace iteration algorithm \cite[Alg.~4.4]{HalMarTro11} with an oversampling parameter $2$ and  $q=2$ is used to compute SVD approximations that are then turned into the QR decompositions \cref{eq:QR1} and \cref{eq:QR2}; see also the offline complexity discussion in \cref{ssec:complexities}. In \cref{fig:LR_spectra}, we present the singular values of both the (adjoint) Helmholtz solution operator and the correlation operator appearing on the right-hand side of \eqref{eq:QR2}. In contrast to the direct data setting—where the truncation error can be analyzed and bounded \cite{BThGhaMarSta13, spantini2015optimal}, working with correlated data complicates this task, because the low-rank approximations enter the posterior covariance matrix \eqref{eq:low_rank_OED:posterior:covariance} and consequently the OED objective (see \cref{thm:low_rank_OED}) repeatedly. However, \eqref{eq:low_rank_OED:posterior:covariance} indicates that the impact of truncation becomes small when the eigenvalues of $\Lh_\w$ are small relative to the unit eigenvalues of the identity. Based on this observation and the spectra in \cref{fig:LR_spectra}, we choose 
$\ell=101$ and $\ellh=1200$, which amounts to retaining all singular values above machine precision in \eqref{eq:QR1} and all singular values larger than $10^{-5}$ in \eqref{eq:QR2}.
We have numerically verified that our results remain unaffected when more singular values are kept. It would clearly be preferable to have a rigorous a priori bound on the truncation error also for the case of correlation data. As shown in \cref{tab:complexity}, the online computational cost increases quadratically with $\ellh$, and consequently so does the running time of \cref{alg:p_cont}; we also observe this behavior empirically in our computations. Finally, the computation of the gradient of the objective was validated against the \textsc{autograd} package \cite{MacDuvAda2015}, which yielded relative errors typically below $10^{-14}$, but required roughly four to six times more computation time.
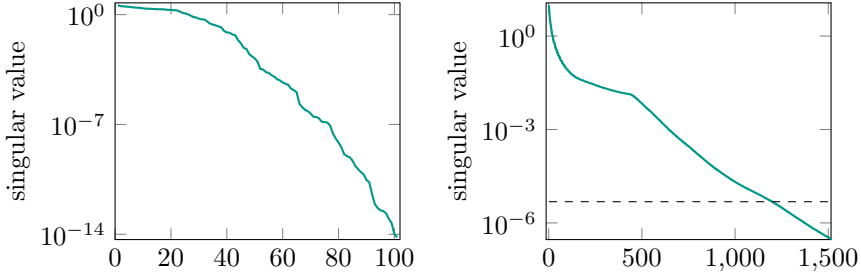
\begin{figure}[bt]
\begin{center}
\begin{tikzpicture}
\begin{groupplot}[scale=0.55,
            group style={
            group size=2 by 1,
            horizontal sep=55pt
            },
]
\nextgroupplot[
	enlargelimits = 0.01,
	legend pos = south west,
    legend style={
        font=\small,
        inner sep=2pt,
        /tikz/every even column/.append style={column sep=3pt}
    },
	ylabel = {singular value},
	ylabel shift = -3pt,
    ymode=log]

\addplot[thick, color=col2] table[x=ind,y=eig, col sep=comma]{CSV/spectrum.csv};

\nextgroupplot[
	enlargelimits = 0.01,
	legend pos = south west,
    legend style={
        font=\small,
        inner sep=2pt,
        /tikz/every even column/.append style={column sep=3pt}
    },
	ylabel = {singular value},
	ylabel shift = -3pt,
    ymode=log,
    xmin=1,
    xmax=1500]

\addplot[thick, color=col2] table[x=ind,y=eig, col sep=comma]{CSV/spectrum_corr.csv};
\addplot[mark=none, black, dashed] coordinates {(1,4.8e-6) (3000,4.8e-6)};
\end{groupplot}
\end{tikzpicture}
\caption{Left: Singular values of first low-rank compression \eqref{eq:QR1}. Right: Singular values of second compression \eqref{eq:QR2}. The dashed line indicates the cutoff we use to retain 1200 singular values.}\label{fig:LR_spectra}
\end{center}
\end{figure}

\subsection{MAP estimation and posterior samples}
Computing optimal designs for linear inverse problems does not require MAP estimates or posterior sampling. Nevertheless, here we briefly present these results to gain insight into the properties of the Helmholtz passive imaging problem and to outline the computational steps, since similar operations are needed for optimal design.

\subsubsection{MAP estimation}
The synthetic data $G$ for our MAP experiments is computed as
\begin{equation*}
  G = \Fc\Diag(q^\dagger)\Fc^* + \Epsilon,
\end{equation*}
where $q^\dagger$ is the ``truth'' covariance diagonal visualized on the left of \cref{fig:domain}. Moreover, for the noise matrix $\Epsilon\sim\Nc(0,\Gmn\kron\Gmn)$, we use a symmetric Gaussian random matrix with entries distributed as $\Nc(0,1)$, multiplied by $\Gmn^{1/2}$ on both sides.

For a fixed design vector $\w$, computing the MAP estimate requires the solution of a large-scale linear system. While \eqref{eq:low_rank_OED:posterior:mean} gives a low-rank expression for the MAP point $q_\post(\w)$, we obtain a higher accuracy by utilizing \cref{prop:posterior_sensor}, albeit replacing the PDE-based map $\Fc$ with its low-rank form from \cref{thm:low_rank_OED}, as these are well approximated by the low-rank formulations and thus speed up computation time.

To calculate the MAP point using \eqref{prop:posterior_sensor:eq:posterior_diagonal} requires computing the $\diag$ operator once for the right hand side and then for each application of $\Cc_{\post}(\w)$. These latter applications are required since the linear system is solved iteratively using the conjugate gradient method implemented in \textsc{SciPy}, in which we employ the low-rank inverse form \eqref{eq:low_rank_OED:posterior:covariance} as a preconditioner. Due to this preconditioning, the conjugate gradient method typically converges in 1–2 iterations.

Consequently, efficient numerical implementation of $\diag$ is essential. Although \cref{lem:diag_form2} and the subsequent remark were key to the low-rank analysis, they are less suitable for explicitly evaluating $\diag$. Indeed, as $n\in\N$ is the discretization size of $X$, this would require $n$ forward evaluations with the operator whose $\diag$ is to be evaluated. Instead, we make use of the second characterization of \cref{lemma:diag}. Although computing an eigensystem may appear to be no better than performing $n$ forward evaluations, the operator appearing inside the $\diag$ in \eqref{prop:posterior_sensor:eq:posterior_diagonal} has rank at most $\min\{\ell,\|\w\|_0\}$. Consequently, only this many eigenpairs must be computed, which renders the method largely independent of the discretization.

\cref{fig:MAP_12} shows on the left the MAP estimate and the location of the 12 sensors used in the Bayesian inversion. How these sensor locations are obtained is discussed in the next section.
Comparing the MAP estimate with the truth $q$ shown on the left in \cref{fig:domain}, we observe that the MAP estimate shows slightly less detail but captures the general behavior well.

\begin{figure}
\centering
\begin{tikzpicture}
\node at (1,1) {\includegraphics[width=0.5\linewidth]{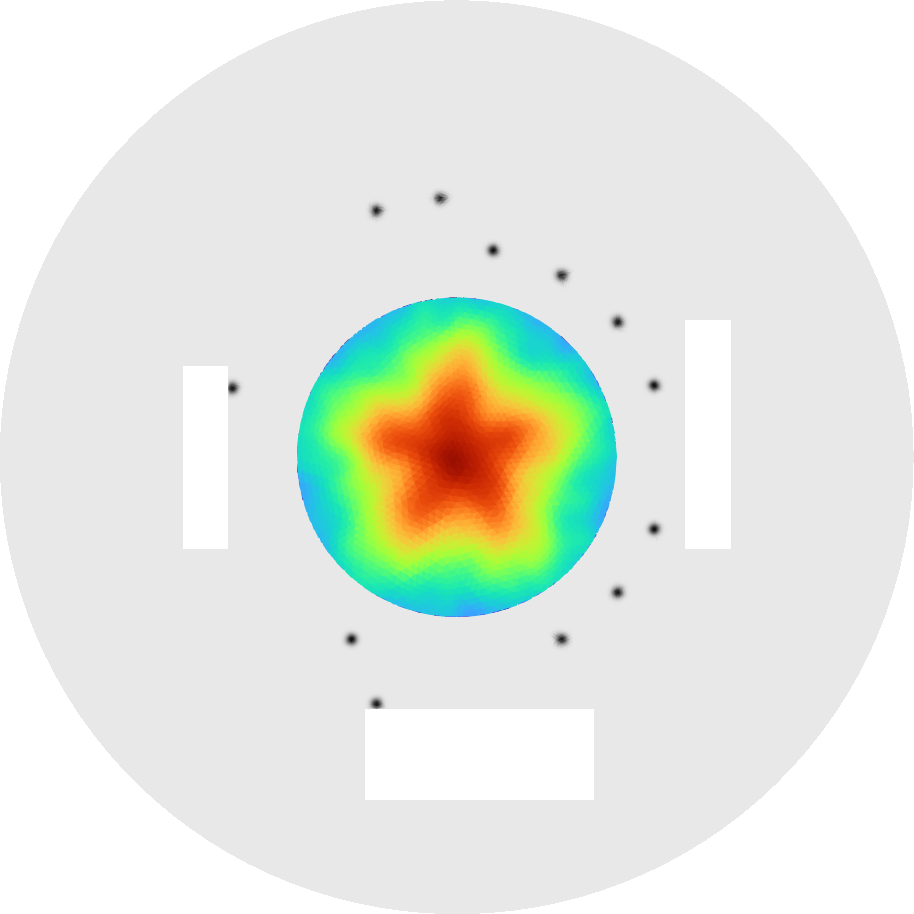}};
\node at (7.5,3.8) {\includegraphics[width=0.3\linewidth]{Graphics/corr_q_colorbar.png}};
\node at (6,2) {\includegraphics[width=0.19\linewidth]{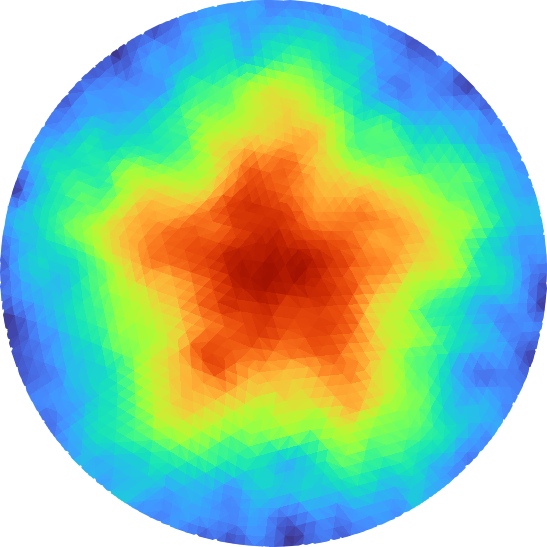}};
\node at (9,2) {\includegraphics[width=0.19\linewidth]{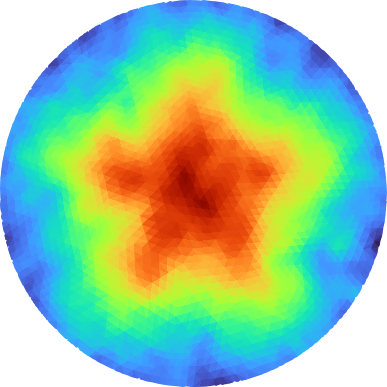}};
\node at (6,-1) {\includegraphics[width=0.19\linewidth]{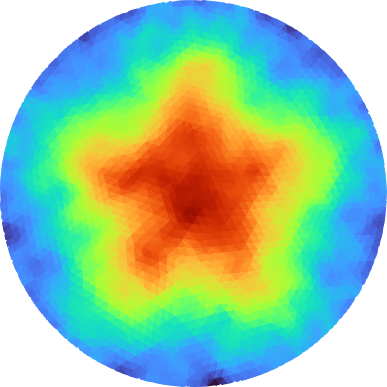}};
\node at (9,-1) {\includegraphics[width=0.19\linewidth]{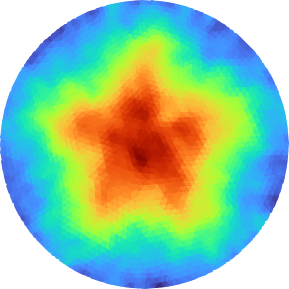}};
\end{tikzpicture}
\caption{Posterior distribution for $\wpcontnom$. Shown on the left is the MAP estimate in $\Omega$ and the 12 sensor locations in $B_1$ used for this inference. Shown in the four smaller figures on the right are samples from the posterior distribution over $\Omega$.} \label{fig:MAP_12}
\end{figure}

\subsubsection{Posterior samples}
We now generate posterior samples for the fixed design $\w$ corresponding to the 12-sensor configuration discussed above, depicted on the right side of \cref{fig:MAP_12}. The variability among these samples reflects the uncertainty in the inferred solution.
Drawing samples from the posterior, we must identify a suitable square root of the matrix representation \eqref{eq:low_rank_OED:posterior:covariance} of the posterior covariance. Adopting an approach analogous to that in \cite[App.~A]{BThGhaMarSta13}, we employ the identity
\begin{align*}
    \Big(
I - \Qh \Qh^T + 
\Qh \big(\Lh_\w  + I_{\ellh}\big)^{-1}\Qh^T
\Big)^{1/2} =
I - \Qh \Qh^T + 
\Qh \left(\Lh_\w  + I_{\ellh}\right)^{-1/2}\Qh^T =:B.
\end{align*}
to establish the following factorization for the posterior covariance matrix \eqref{eq:low_rank_OED:posterior:covariance}:
\begin{equation}\label{eq:post-variance-split}
C_\post(\w) = M^{-1}CM^{-1/2}B M B^TM^{-1/2}C.    
\end{equation}
The vector representation of posterior samples $q_\mathrm{s}\in V_h^0$ can thus be written as
\begin{equation}\label{eq:post-sample}
    \q_\mathrm{s} = \q_\post(\w) + 
    M^{-1}CM^{-1/2} B \boldsymbol\eta,
\end{equation}
where $\boldsymbol \eta$ is a vector whose components are random, independent, and identically distributed Gaussian, and the white-noise factor $M^{-1/2}$ exactly offsets the $M^{1/2}$ arising from the mass matrix located between $B$ and $B^T$ in \eqref{eq:post-variance-split}. These samples (as well as samples from the prior) may take negative values, which prevents us from interpreting $q_s(\x)$ as the pointwise variance of a diagonal operator. Nevertheless, negative entries in $q$ do not pose an issue for the forward (and hence the inverse) problem when it is expressed in the form \eqref{eq:ip_diagonal_sensor}.

\subsection{Computation of optimal designs}\label{sec:Helmholtz:OED-compute}


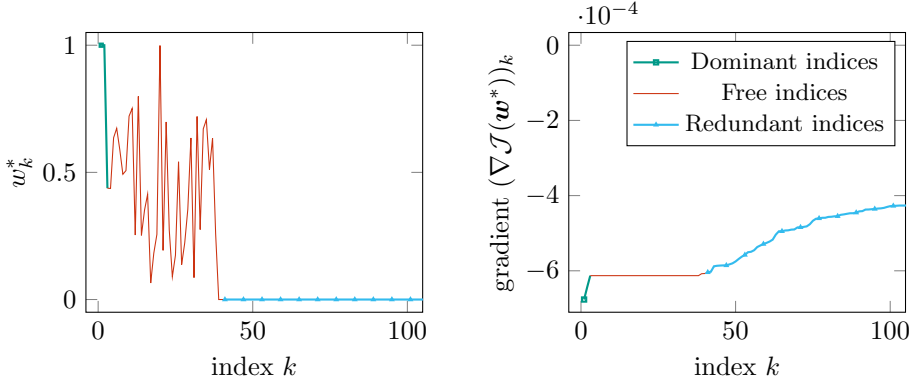
\begin{figure}[bt]
\begin{center}
\begin{tikzpicture}
\begin{groupplot}[scale=0.65,
            group style={
                group size=2 by 1,
                horizontal sep=55pt
            },
]
\nextgroupplot[
	ymin = 0,
	ymax = 1,
    xmax = 100,
	enlargelimits = 0.05,
	xlabel = {index $k$}, 
	ylabel = {$w^*_k$},
	ylabel shift = -3pt,
    ]
	
\addplot[thick, color=col2,mark=square,mark size=0.6pt,mark repeat=6] table[x=ind,y=w, col sep=comma]{CSV/w1_018_dom.csv};
\addplot[thin, color=col3] table[x=ind,y=w, col sep=comma]{CSV/w1_018_free.csv};
\addplot[thick, color=col5,mark=triangle,mark size=0.5pt,mark repeat=6] table[x=ind,y=w, col sep=comma]{CSV/w1_018_red.csv};

\nextgroupplot[
	ymax = 0,
    xmax = 100,
	enlargelimits = 0.05,
	legend pos = north east,
    legend style={
        font=\small,
        inner sep=2pt,
        /tikz/every even column/.append style={column sep=3pt}
    },
	xlabel = {index $k$}, 
	ylabel = {gradient $(\nabla\Jc(\w^*))_k$},
	ylabel shift = -3pt]

\addplot[thick, color=col2,mark=square,mark size=0.8pt,mark repeat=3, thick] table[x=ind,y=jac, col sep=comma]{CSV/w1_018_dom.csv};
\addlegendentry{Dominant indices}
\addplot[thin, color=col3] table[x=ind,y=jac, col sep=comma]{CSV/w1_018_free.csv};
\addlegendentry{Free indices}
\addplot[thick, color=col5,mark=triangle,mark size=0.5pt,mark repeat=6] table[x=ind,y=jac, col sep=comma]{CSV/w1_018_red.csv};
\addlegendentry{Redundant indices}
\end{groupplot}
\end{tikzpicture}
\caption{Left: $1$-relaxed optimal design $\w^*$ using $m_0=18$ out of $m=344$ sensors. Right: corresponding gradient. Dominant indices ($w^{*}_k=1$, resp.~large negative gradient) as green squares, redundant indices ($w^{*}_k=0$, resp.~small negative gradient) as cyan triangles, free indices in red. }\label{fig:w_star_1}
\end{center}
\end{figure}

In \cref{fig:w_star_1}, we show the non-binary optimum $\w^{*}$ of \eqref{eq:OED} (with indices sorted in increasing gradient order) for $p=1$, $m_0=18$ as computed by \cref{alg:p_cont}. Although non-convexity of the objective means we cannot guarantee that $\w^{*}$ is a global optimum, we observe analogous behavior of the gradient as in the global optimality criterion \cite[Thm.~3]{Aarset2024} for convex OED objectives. Namely, dominant indices ($w_k^{*}=1$) correspond to large negative gradient entries, redundant indices ($w^{*}_k=0$) correspond to small negative gradient entries, and the remaining gradient entries are constant. This pattern arises even though there are no theoretical guarantees in the passive imaging setting. This indicates that, despite the non-convex nature of the problem, $\w^{*}$ serves as an excellent initialization for approximating binary optimal designs. In parallel, \cref{tab:sparsity} reports the number of $p$-sparsification iterations required by \cref{alg:p_cont} and the resulting count of non-zero indices as a function of $p$.

\begin{table}[tb]
    \caption{Shown are the counts of non-zero components of $\w^{(i)}$ and the associated values of $p$ for each iteration $i$ of \cref{alg:p_cont}. The desired number of active sensors is $m_0 = 18$ out of $m = 344$ possible locations. The algorithm stops after 6 iterations.
}
    \centering
    \begin{tabular}{c|c|c|c|c|c|c}
        iter $i$ & 0 & 1 & 2 & 3 & 4 & 5 \\ \hline
        $p$ & 1.0 &0.9 & 0.81 & 0.73 & 0.66 & 0.59 \\ \hline
        $\|\w^{(i)}\|_0$ & 40 & 32 &21 &20& 20 & 18
    \end{tabular}
    \label{tab:sparsity}
\end{table}

\begin{figure}[bth]  
\centering
\includegraphics[width=.4\textwidth,keepaspectratio]{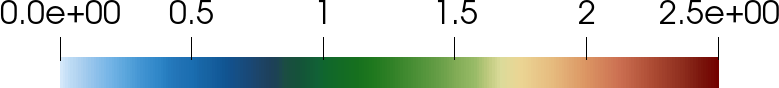}\\[2ex]
\includegraphics[width=.32\textwidth,keepaspectratio]{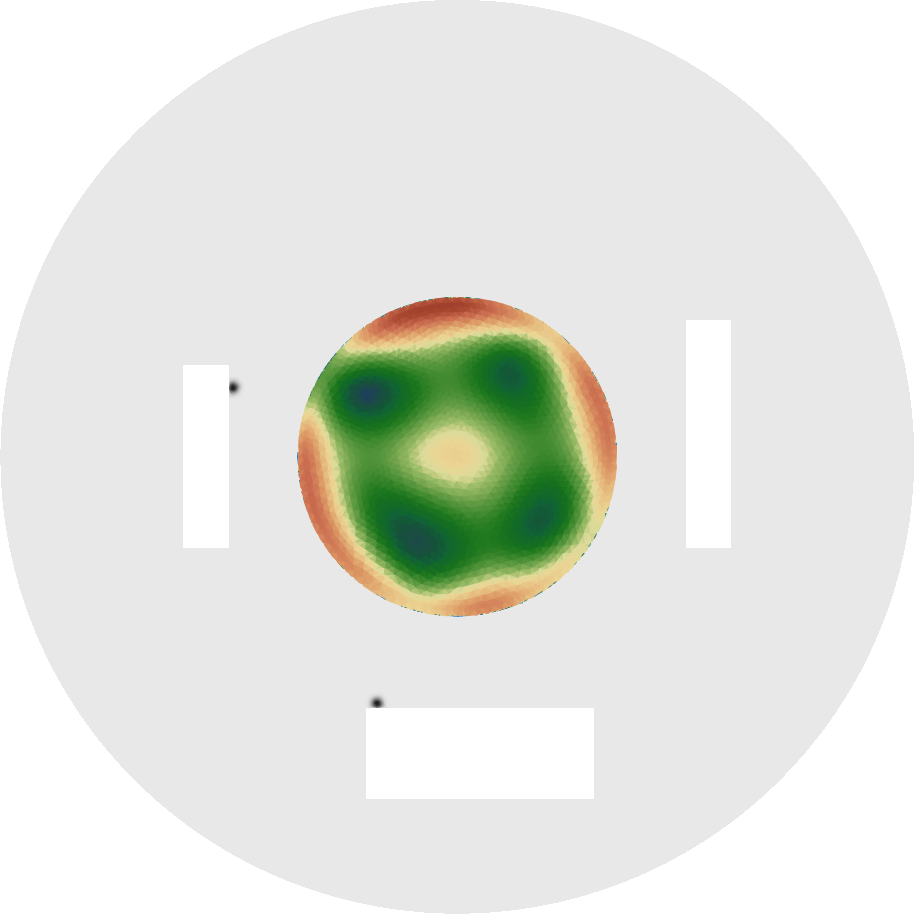}
\includegraphics[width=.32\textwidth,keepaspectratio]
{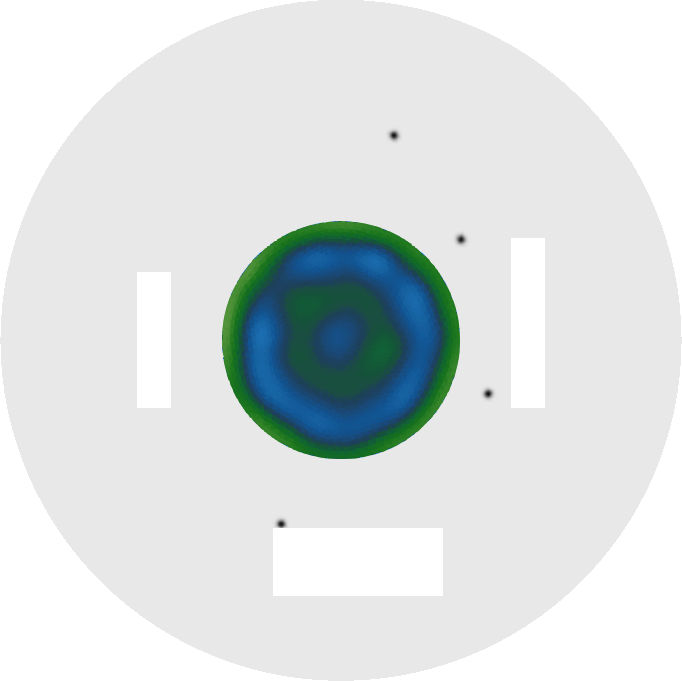}
\includegraphics[width=.32\textwidth,keepaspectratio]
{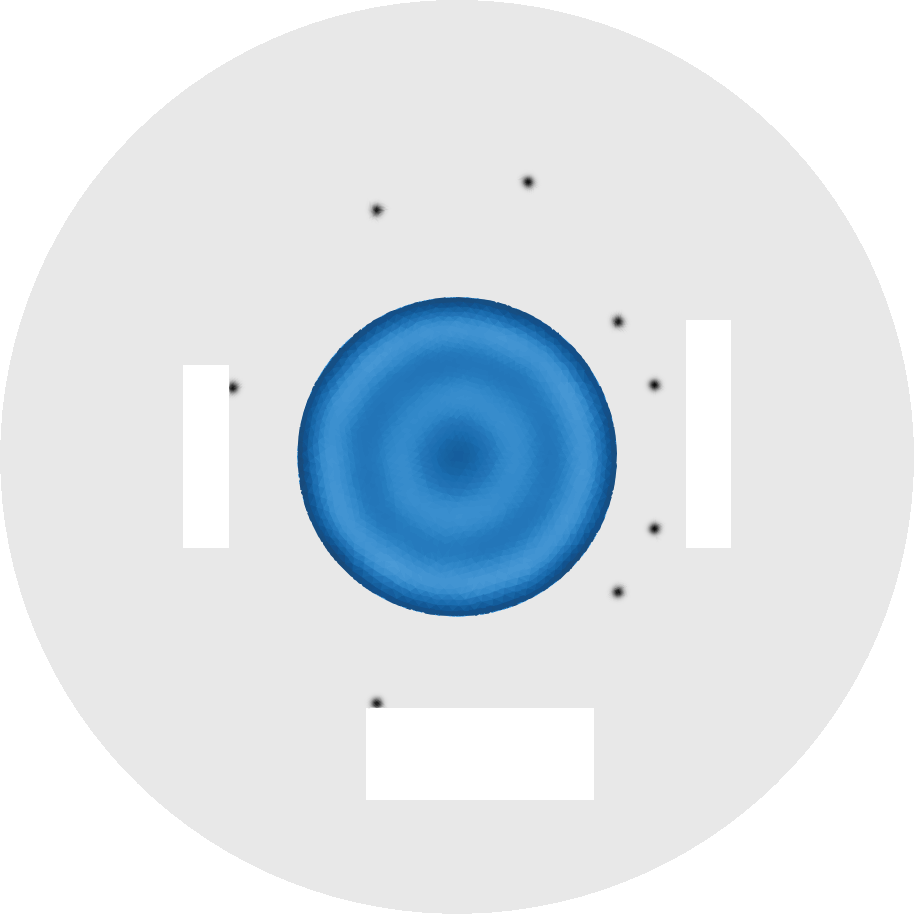}
\caption{Optimized sensor locations (grey dots) and corresponding posterior pointwise variance fields (inner circles). Left to right: $m_0\in\{2, 4, 8\}$.
}\label{fig:designs_various}
\end{figure}

\subsection{Discussion of optimal designs}\label{sec:Helmholtz:OED-discussion}
Next, we study the properties of different designs.  In \cref{fig:designs_various}, we show various sensor designs obtained with our algorithm, as well as the pointwise posterior variance fields for each design. Generally, we observe that the pointwise variance decreases rapidly as the number of sensors increases. The algorithm does not only select sensors close to the source domain $\Omega$. Additionally, the chosen sensor locations for 2 sensors are not a subset of the sensor locations obtained for 4 sensors, as would be the case when a greedy algorithm was used.

\begin{figure}[tbh]
\centering
\begin{tikzpicture}[scale=1]
\begin{axis}[
	no markers, 
        xmin=1,
        ymin = 0.0,
        ymax = 0.8,
	enlargelimits = 0.01, 
	xlabel = {number of sensors $m_0$}, 
	xlabel style = {font = \normalsize}, 
	ylabel = {$\Jc(\w)$}, 
	ylabel style = {font = \normalsize}, 
	ylabel shift=-5pt, 
	xlabel shift=-3pt, 
	xshift=-12pt,
        width = 6.8cm,
        legend cell align={left},
        legend columns=2,
	legend style={at={(0.5,1.3)},anchor=north,
                /tikz/column 2/.style={
                column sep=5pt}
            }
]
\addplot[color=red!40, name path=maxes, thick, forget plot] table[x=targets,y=randommax, col sep=comma]{CSV/Aoptimalities.csv};
\addplot[color=red!40, name path=mins, thick, forget plot] table[x=targets,y=randommin, col sep=comma]{CSV/Aoptimalities.csv};
\addplot[color=OrangeRed!40,opacity=0.4]fill between[of=maxes and mins];
\addlegendentry{random}
\addplot[very thick, color=blue, very thick, dashed, name path=w1s] table[x=targets,y=w1, col sep=comma]{CSV/Aoptimalities.csv};
\addlegendentry{non-bin.\ {$\w^*$}}
\addplot[very thick, color=black!70, very thick, dotted, name path=greedies] table[x=targets,y=greedy, col sep=comma]{CSV/Aoptimalities.csv};
\addlegendentry{greedy}
\addplot[very thick, color=ForestGreen!60, very thick, name path=ws] table[x=targets,y=w, col sep=comma]{CSV/Aoptimalities.csv};
\addlegendentry{Alg.~\ref{alg:p_cont}, $\wpcontnom$}
\end{axis}
\end{tikzpicture}\hfill
\begin{tikzpicture}
\begin{axis}[
	no markers, 
        xmin=1,
	enlargelimits = 0.01, 
	xlabel = {number of sensors $m_0$}, 
	xlabel style = {font = \normalsize}, 
	ylabel = {relative diff to $\w^*$ in \%}, 
	ylabel style = {font = \normalsize}, 
	ylabel shift=-5pt, 
	xlabel shift=-5pt, 
	xshift=-12pt,
        width=6.8cm,
	legend style={font=\scriptsize, at={(0.05,1.28)},anchor=north west}
]
\addplot[very thick, color=ForestGreen, very thick, name path=ws] table[x=targets,y expr=100*(\thisrow{w}-\thisrow{w1})/\thisrow{w1}, col sep=comma]{CSV/Aoptimalities.csv};
\addlegendentry{$(\Jc(\wpcontnom)-\Jc(\w^*))/\Jc(\w^*)$}

\addplot[very thick, color=black!70, dotted, very thick, name path=ws] table[x=targets,y expr=100*(\thisrow{greedy}-\thisrow{w1})/\thisrow{w1}, col sep=comma]{CSV/Aoptimalities.csv};
\addlegendentry{$(\Jc(\wgreedynom)-\Jc(\w^*))/\Jc(\w^*)$}

\end{axis}
\end{tikzpicture}
\caption{Left: Comparison of A-optimal objective $\Jc(\w)$ vs.\ number of sensors for designs $\wpcontnom$ computed with \cref{alg:p_cont} (green), with a greedy algorithm (black, dotted) and the non-binary designs $\w^*$ (blue, dashed) used to initialize \cref{alg:p_cont}.
Shown for comparison are results for~$10^3$ random designs (red). 
Right: Relative difference (in \%) of optimized binary designs $\wpcontnom$ and greedy designs $\wgreedynom$ to non-binary designs $\w^*$. } 
\label{fig:comps}
\end{figure}
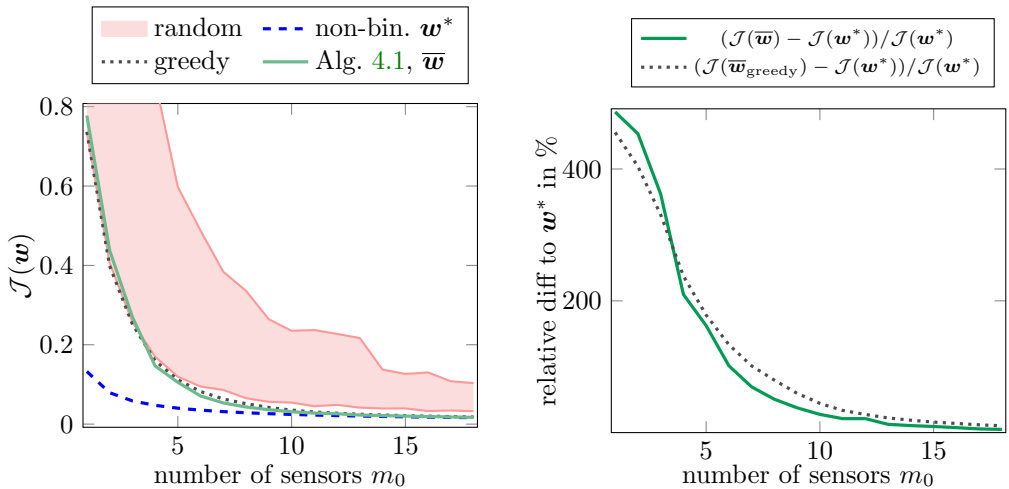

On the left in \cref{fig:comps}, we show $\Jc(\w)$ for different numbers of sensors for designs obtained with the $p$-continuation algorithm \ref{alg:p_cont}, designs found through greedy optimization, and non-binary designs discussed above. We find that the differences between these optimized designs are small. 
The figure also contains a comparison with random designs chosen from among all candidate locations. Random designs perform substantially worse than optimized designs. In \cite{Aarset2024} random designs were found to be more competitive when restricted to only a subset of sensors in the inner circles of possible locations. This effect was less pronounced for the current experiment; the poorest-performing \enquote{circular random designs} still outperformed the poorest fully random designs by a significant margin, whereas the top \enquote{circular random designs} did not show a significant advantage over the top fully random designs. A plausible reason is that allowing locations to be selected freely at random may enhance the chance of obtaining beneficial data correlations between sensors. On the right in \cref{fig:comps}, we show the relative difference between the optimized and greedy designs and the non-binary designs that sum to the same value. Here, the values obtained with the non-binary designs are lower bounds. As discussed in \cref{sec:OED}, this is not guaranteed by theory, as the optimization problem for the non-binary designs \eqref{eq:OED} is not necessarily convex. This is a difference from the direct measurements case, where the non-binary designs minimize a convex functional and are thus guaranteed to be a lower bound of the objective.
Note that for moderate $m_0$, the non-binary objective $\mathcal{J}(\mathbf{w}^*)$ is small in absolute terms, so even a small absolute gap between binary and non-binary designs becomes large in relative terms.

Finally, we qualitatively compare the designs shown in \cref{fig:designs_various} for correlation data with the designs obtained for \emph{direct} data in \cite[Fig.\ 6]{Aarset2024}, which uses a similar setup. The designs for direct observations preferably select sensor locations from the inner rings of candidate locations, i.e., those close to $\Omega$. This is in contrast to the designs we find for correlation data, which include---even for moderate numbers of sensors---more scattered locations further away from $\Omega$. This is consistent with nearby sensors observing largely the same wave field, so the off-diagonal entries of their correlation matrix carry less independent information.

\section{Conclusion and outlook}
We have addressed the challenging problem of optimal experimental design for the passive imaging problem \eqref{eq:ip_diagonal_sensor} with a spatially uncorrelated source. Our main findings and contributions are:
(1) A Bayesian analysis of the 
    posterior distribution of the source strength given correlation data (\cref{prop:posterior_sensor}) and the accompanying formulation of the optimal experimental design problem with the A-optimal objective \eqref{eq:A_optimal}.
(2) A two-level low-rank formulation of the above posterior distribution and A-optimal objective, including a computationally efficient expression for the gradient in \cref{thm:low_rank_OED} and the associated computational complexity analysis (\cref{ssec:complexities}).
(3) A computational scheme for the discretization of the unknown source strength and the source-strength-to-correlation-data mapping (\cref{ssec:FEM}), and an accompanying numerical study of optimal designs for passive imaging with the Helmholtz equation, suitable for determining optimal experimental designs in an aeroacoustic setting.

The computational framework and low-rank posterior representation are generally applicable beyond optimal experimental design. As shown in the numerical experiments, the method can draw posterior samples of the unknown source strength $q^\dagger$ even with few measurement locations, making it suitable for a broad class of passive imaging problems. Extending the framework to spatially correlated sources, i.e., reconstructing the full source covariance $\Qc^\dagger$ rather than only its diagonal, is an interesting direction for future work; however, its high dimension is challenging. Extending passive imaging to nonlinear parameter identification, that is, using \eqref{eq:IP_uncorrelated} with a nonlinear forward map $\Fc$, would broaden the scope of this work; however, optimal experimental design for nonlinear inverse problems is already highly challenging, even for direct imaging \cite{AlePetStaGha16, ChowdharyAttiaAlexanderian2025, KovalNicholson2025}.

\bibliographystyle{siam}
\bibliography{bibliography}

\appendix
\section{Stochastic aspects}\label{app:Wishart}
As discussed in \cref{ssec:related_literature}, it is known that correlation data follows a \emph{Wishart distribution}, rather than a Gaussian one. Although this distribution was originally introduced in \cite{Wis1928}, our presentation here follows the notation and derivations in \cite{Mui82}.
For $N\in\N$, let $(\g^{(i)})_{i=1}^N$ be i.i.d.~Gaussian random vectors in $\R^m$, $m\in\N$, with mean $0$ and covariance matrix $\Gamma\in\R^{m\times m}$. The resulting $m\times m$ (unscaled) sample covariance matrix is then said to follow a Wishart distribution with $N$ degrees of freedom and covariance matrix $\Gamma$, which we denote
$
    \sum_{i=1}^N \g^{(i)}\corr \g^{(i)} \sim W_m(N,\Gamma).
$
We now summarize fundamental properties of Wishart distributions and refer to \cite[Thm.~3.1.4]{Mui82} and the preceding discussion for proofs.
\begin{lemma}\label{lemma:wishart_properties}
Let $G_N\sim W_m(N,\Gamma)$, and let $K_{m^2}\in\R^{m^2\times m^2}$ be the unique matrix satisfying $K_{m^2}\vec(H)=\vec(H^T)$ for all $H\in\R^{m\times m}$. Then:
\begin{enumerate}
    \item $G_N$ is positive definite with probability $1$ if and only if $N\geq m$. \label{lemma:wishart_properties:posdef}
    \item $\E[G_N] = N\Gamma$, that is, $G_N/N$ is an unbiased estimator of the original covariance matrix $\Gamma$. \label{lemma:wishart_properties:E}
    \item $\Cov[\vec(G_N)] = N(I_{m^2}+K_{m^2})(\Gamma\kron\Gamma)$. \label{lemma:wishart_properties:Cov}
    \item The asymptotic distribution of $\sqrt{N}\vec(\frac{1}{N}G_N-\Gamma)$ as $N\to\infty$ is $\Nc(0,(I_{m^2}+K_{m^2})(\Gamma\kron\Gamma)$. \label{lemma:wishart_properties:normal}
\end{enumerate}
In particular, if $\vec(G_N)$ is treated as belonging to the ${m(m+1)}/{2}$-dimensional subspace of $\R^{m^2}$ corresponding to symmetric matrices, i.e., the range of $I_{m^2}+K_{m^2}$, where this transformation is identical to $2I_{m^2}$, then $\Cov[\vec(G_N)] = 2N(\Gamma\kron\Gamma)$. 
\end{lemma}
Taken together, items \ref{lemma:wishart_properties:posdef}-\ref{lemma:wishart_properties:Cov} above show that $G_N/N$, typically referred to as the \emph{sample covariance matrix}, serves as an estimator of the covariance $\Gamma$ of each original sample $\g^{(i)}$ as $N$ increases. In addition, item \ref{lemma:wishart_properties:posdef}, together with the concluding remark, implies that we may effectively omit the correction term $I_{m^2}+K_{m^2}$ in the covariance expression \eqref{lemma:wishart_properties:Cov}, provided we implicitly restrict our attention to the space $\R^{m\times m}_{\mathrm{symm}}$ of symmetric matrices, i.e., the image of $I_{m^2}+K_{m^2}$.

\begin{lemma}\label{corr:wishart_inverse}
Consider the inverse problem \eqref{eq:IP_uncorrelated_sensor}, with $f\sim\Nc(0,Q)$ and $\eps\sim\Nc(0,\Gmn)$, $Q\in\HS(X^*,X)$, $\Gmn\in\R^{m\times m}$. If $(f^{(i)})_{i=1}^N$ and $(\eps^{(i)})_{i=1}^N$ are $N$ i.i.d.~realizations of these distributions, then
\[
    G_N:=\sum_{i=1}^N \g^{(i)}\corr\g^{(i)} \sim W_m(N,\Fc_\w\Diag(q)\Fc_\w^*+\Gmn)
\]
and in particular, 
\begin{align*}
    \E[G_N/N] & = \Fc_\w\Diag(q)\Fc_\w^* + \Gmn, \\
    \Cov[G_N/N] & = 2N^{-1}\left(
        \left(\Fc_\w\Diag(q)\Fc_\w^* + \Gmn\right) \kron
        \left(\Fc_\w\Diag(q)\Fc_\w^* + \Gmn\right)
        \right)=:\Sigma_{N,\w,q}.
\end{align*}
Asymptotically, as $N$ grows, $\Sigma_{N,\w,q}$ decays to $0$, and $G_N/N-\Gmn$ is approximately distributed as $\Nc(\Fc_\w\Diag(q)\Fc_\w^*,\Sigma_{N,\w,q})$, in the sense that if 
\[
    G\sim\Nc(\Fc_\w\Diag(q)\Fc_\w^*,\Sigma_{N,\w,q})
\]
on $\R^{m\times m}_{\mathrm{symm}}$, then for all convex sets $U\subset\R^{m\times m}_{\mathrm{symm}}$, there is some $C>0$ so that one has
\begin{gather*}
    \left|
        \P\left[
        G_N/N-\Gmn\in U\right] -
        \P\left[
        G\in U\right]
    \right| \leq \\
    C\left(\frac{m(m+1)}{2}\right)^{1/4}\sum_{i=1}^N\E\left[\|\Sigma_{N,\w,q}^{-1/2}(\g^{(i)}\otimes\g^{(i)}-\Gmn)\|_2^3\right].
\end{gather*}
\end{lemma}
\begin{proof}
The only novel aspect is the final claim, which is an immediate consequence of the multivariate Berry-Esseen theorem as stated in \cite[Thm.~1.1]{Ben05}.
\end{proof}
This lemma provides partial justification for the Gaussian approximation adopted throughout the paper and highlights an important observation: since we work with intrinsically noisy data $\g^{(i)}$ \emph{before} forming correlations, it is insufficient to simply compute the standard sample covariance $\frac{1}{N}\sum_{i=1}^N\g^{(i)}\kron\g^{(i)}$. The expectation of this estimator is $\Fc_\w\Diag(q)\Fc_\w^*+\Gmn$, where $\Gamma$ \emph{does not} vanish as $N$ grows. When the noise level is non-negligible, this leftover term introduces bias if used directly for inversion. Consequently, to obtain meaningful covariance information from individual noisy measurements, one should instead work with the noise-corrected sample covariance matrix $\big(\frac{1}{N}\sum_{i=1}^N\g^{(i)}\kron\g^{(i)}\big) - \Gmn$.

Note that the covariance matrix $\Sigma_{N,\w,q}$ appearing in the above lemma inherently depends on both $\w$ and $q$. Moreover, even though in \cref{sec:passive_imaging} we suppose that the Gaussian approximation $G$ of the covariance data has a covariance that can be written as a Kronecker product of two diagonal matrices, there is, in general, no reason to expect that $\Sigma_{N,\w,q}$ possesses this structure. Indeed, $\Sigma_{N,\w,q}$ is determined by the matrix $\Fc_\w\Diag(q)\Fc_\w^*$, which is typically non-diagonal. To the best of the authors’ knowledge, however, there is no closed-form expression for the posterior distribution in the Bayesian inverse problem \eqref{eq:ip_diagonal_sensor}. This observation supports the methodology adopted here, which is shown in \cref{sec:Helmholtz} to perform robustly in practice.

\section{Linear algebra identities}\label{app:linear_algebra}
Here, we collect the matrix identities used in Theorem \ref{thm:low_rank_OED} and in \cref{ssec:complexities}. The first and fourth equalities in the next lemma are classical, found in e.g.~\cite[Lem.~4.3.1, Lem.~5.1.3]{HorJoh1991}, while the second and third equalities can be deduced in a similar fashion.
\begin{lemma}[Fast matrix-vector applications]\label{lemma:linear_algebra}
For $n_1$, $n_2\in\N$ and given $\x\in\R^{n_1}$, $\y\in\R^{n_2^2}$, $A\in\R^{n_2\times n_1}$, $B\in\R^{n_1\times n_2}$, the following identities hold:
\begin{align*}
    [A\schur B]\x & = \diag\left(A\Diag(\x)B^T\right) \in \R^{n_2}, \\
    [A\khra B]\x & = \vec\left(
        A\Diag(\x)B^T
    \right) \in \R^{n_2^2}, \\
    [A^T\khra B^T]^T\y & = \diag\left(
        A\mat(\y)B^T
    \right) \in \R^{n_1}, \\
    [A\kron B]\y & = \vec\left(
        A\mat(\y)B^T
    \right) \in \R^{n_1^2}.
\end{align*}
Here, the reshaping operator $\mat$ maps the entries of a vector of square length to the entries of a square matrix (filling column-wise from the first entry). The opposite operator from matrix to vector is denoted by $\vec$.
\end{lemma}
Using these identities instead of forming matrix products explicitly can substantially decrease computational costs.
\begin{lemma}[Complexity of matrix-vector products]\label{corr:linear_algebra_complexities}
Under the assumptions of the preceding lemma, the following results are valid:
\begin{enumerate}
    \item $[A\schur B]\x \in \R^{n_2}$ can be computed in $\mathcal O(n_1n_2^2)$ time. \label{corr:linear_algebra_complexities:schur}
    \item $[A\khra B]\x\in\R^{n_2^2}$ can be computed in $\mathcal O(n_1n_2^2)$ time. \label{corr:linear_algebra_complexities:column}
    \item $[A^T\khra B^T]^T\y\in\R^{n_1}$ can be computed in $\mathcal O(n_1n_2(n_1+n_2))$ time. \label{corr:linear_algebra_complexities:row}
    \item $[A\kron B]\y\in\R^{n_1^2}$ can be computed in $\mathcal O(n_1n_2(n_1+n_2))$ time. \label{corr:linear_algebra_complexities:kron}
\end{enumerate}
Moreover, none of the matrix-matrix products $A\schur B$, $A\khra B$, $A^T\khra B^T$, or $A\kron B$ need to be assembled to obtain these computational complexities.
\end{lemma}
The following lemmas are essential for the two-level decompositions \eqref{eq:QR1}, \eqref{eq:QR2} and for the computation of the objective gradient in Theorem~\ref{thm:low_rank_OED}.

\begin{lemma}[Schur product decomposition \cite{Sly1999}]\label{lem:schur_decomp}
For arbitrary $n_1$, $n_2$, $n_3\in\N$, let $A$, $D\in\R^{n_1\times n_2}$, $B$, $E\in\R^{n_2\times n_2}$, $C$, and $F\in\R^{n_2\times n_3}$. Then
\[
	\left[
		ABC
	\right] \schur
	\left[
		DEF
	\right] =
	\left[
		A^T \khra D^T
	\right]^T
	\left[
		B \kron E
	\right]
	\left[
		C \khra F
	\right],
\]
recalling from \cref{ssec:notation} that $\khra$ denotes the column-wise Khatri-Rao product.
\end{lemma}

\begin{lemma}[Trace of Kronecker products]\label{lem:Kronecker_trace}
For $n_1$, $n_2\in\N$ , let $A$, $B\in\R^{n_1\times n_1}$ and $C\in\R^{n_2\times n_1^2}$, $D\in\R^{n_1^2\times n_2}$ be generic real matrices. Then
\begin{align*}
    \tr\left(C\left[A\kron B\right]D\right) & = 
    \tr\bigg(A\bigg[
    \sum_{i=1}^{n_2} \mat(D_{:i})^TB^T\mat(C_{i:})
    \bigg]\bigg) \\
    & = \tr\bigg(B
    \bigg[
        \sum_{i=1}^{n_2} \mat(D_{:i})A^T\mat(C_{i:})^T
    \bigg]\bigg).
\end{align*}
Here, $C_{i:}$, $D_{:i}\in\R^{n_1^2}$ represent the $i$-th row of $C$ and the $i$-th column of $D$.
\end{lemma}

\begin{proof}
We begin by noting that $E:=DC=\sum_{i=1}^{n_2}D_{:i}C_{i:}^T\in\R^{n_1^2\times n_1^2}$, where $C_{i:}$ and $D_{:i}$ are understood as $n_1^2$-vectors. By linearity and cyclicity of the trace, we have
\begin{align*}
    \tr\left(
        C\left[A\kron B\right]D
    \right) & = 
    \tr\left(
        \left[A\kron B\right]DC
    \right) =
    \sum_{i=1}^{n_2}\tr\left(
        C_{i:}^T\left[A\kron B\right]D_{:i}
    \right) \\
    & = \sum_{i=1}^{n_2}
        C_{i:}^T\left[A\kron B\right]D_{:i}
    = \sum_{i=1}^{n_2}
        C_{i:}^T\vec\left(A\mat(D_{:i})B^T\right) \\
    & = \sum_{i=1}^{n_2}
        \tr\left(
            \mat(C_{i:})^T A\mat(D_{:i})B^T
        \right).
\end{align*}
Cyclicity of the trace and the fact that $\tr(\cdot)=\tr(\cdot^T)$ yields the claim.
\end{proof}

\end{document}